\documentclass[12pt,reqno]{amsart}
\usepackage[notref,notcite]{}
\usepackage{amssymb,verbatim}
\usepackage{amsmath}

\newtheorem{theorem}{Theorem}[section]
\newtheorem{lemma}[theorem]{Lemma}

\newtheorem{corollary}[theorem]{Corollary}

\numberwithin{equation}{section}

\def\C{\mathbb C}
\def\R{\mathbb R}
\def\S{\mathbb S}

\def\l{\langle}
\def\r{\rangle}

\newcommand{\I}{\textsl{i}}
\newcommand{\md}{\mathrm{d}}

\newcommand{\wh}{\widehat}

\newcommand{\lb}{\left(}

\newcommand{\vp}{\varphi}

\newcommand{\rb}{\right)}

\newcommand{\Sc}{\mathcal{S}}

\newcommand{\Rb}{\mathbb{R}}
\newcommand{\Sb}{\mathbb{S}}

\newcommand{\Beq}{\begin{equation}}
\newcommand{\Eeq}{\end{equation}}
\newcommand{\beq}{\begin{equation*}}
\newcommand{\eeq}{\end{equation*}}
\newcommand{\bal}{\begin{align}}
\newcommand{\eal}{\end{align}}

\newcommand{\bpr}{\begin{proof}}
	\newcommand{\epr}{\end{proof}}

\newcommand{\bel}[1]{\begin{equation}\label{#1}}
\newcommand{\ee}{\end{equation}}

\begin{document}
	
	\title[Momentum ray transforms]{Momentum ray transforms}
	\author[V. P. Krishnan, R. Manna, S. K. Sahoo and V. Sharafutdinov]{Venkateswaran P.\ Krishnan$^\ast$$^\dagger$, Ramesh Manna$^\dagger$, Suman Kumar Sahoo$^\dagger$ and Vladimir A.\ Sharafutdinov$^\sharp$}
	
	
	\thanks{The first author was supported by US NSF grant DMS 1616564.}
	\thanks{The second author was supported by SERB National Postdoctoral fellowship, PDF/2017/002780.}
	\thanks{First three authors were supported by  Airbus Corporate Foundation Chair grant
		``Mathematics of Complex Systems'' established at TIFR CAM and TIFR ICTS,
		Bangalore, India.}
	\thanks{The work was started when the last author visited TIFR CAM January 2017. The author is grateful to the institute for the support and hospitality.}
	\thanks{The last author was supported by RFBR, Grant 17-51-150001.}
	
	\email{vkrishnan@tifrbng.res.in,ramesh@tifrbng.res.in,suman@math.tifrbng.res.in,\newline\indent sharaf@math.nsc.ru}
	\address{$^\ast$Corresponding author
		\newline\indent$^\dagger$TIFR Centre for Applicable Mathematics, Sharada Nagar, Chikkabommasandra,\newline\indent\hspace{0mm} Yelahanka New Town, Bangalore, India
		\newline\indent$^\sharp$Sobolev Institute of Mathematics; 4 Koptyug Avenue, Novosibirsk, 630090, Russia;
		\newline\indent\hspace{0mm} Novosibirsk State University, 2 Pirogov street, 630090, Russia}

	
	
	\begin{abstract}
		The momentum ray transform $I^k$ integrates a rank $m$ symmetric tensor field $f$ over lines with the weight $t^k$:
		$
		(I^k\!f)(x,\xi)=\int_{-\infty}^\infty t^k\l f(x+t\xi),\xi^m\r\,\md t.
		$
		In particular, the ray transform $I=I^0$ was studied by several authors since it had many tomographic applications. We present an algorithm for recovering $f$ from the data $(I^0\!f,I^1\!f,\dots, I^m\!f)$. In the cases of $m=1$ and $m=2$, we derive the Reshetnyak formula that expresses $\|f\|_{H^s_t({\R}^n)}$ through some norm of $(I^0\!f,I^1\!f,\dots, I^m\!f)$. The $H^{s}_{t}$-norm is a modification of the Sobolev norm weighted differently at high and low frequencies. Using the Reshetnyak formula, we obtain a stability estimate.
	\end{abstract}
	
	\maketitle
	\section{Introduction}
	
	The ray transform integrates symmetric tensor fields over straight lines. Let $\l\cdot,\cdot\r$ be the standard dot-product on ${\mathbb R}^n$ and $|\cdot|$, the corresponding norm.
	The family of oriented straight lines in ${\mathbb R}^n$ is parameterized by points of the manifold
	$$
	T{\mathbb S}^{n-1}=\{(x,\xi)\in{\mathbb R}^n\times{\mathbb R}^n\mid |\xi|=1,\langle x,\xi\rangle=0\}\subset{\mathbb R}^n\times{\mathbb R}^n
	$$
	that is the tangent bundle of the unit sphere ${\mathbb S}^{n-1}$. Namely, a point $(x,\xi)\in T{\mathbb S}^{n-1}$ determines the line $\{x+t\xi\mid t\in{\mathbb R}\}$. Along with the space $C^\infty(T{\mathbb S}^{n-1})$ of smooth functions, we use the Schwartz space ${\mathcal S}(T{\mathbb S}^{n-1})$. Observe that the space ${\mathcal S}(E)$ is well defined for a smooth vector bundle $E\rightarrow M$ over a compact manifold $M$.
	
	Let $S^{m}\Rb^{n}$ be the complex vector space of rank $m$ symmetric tensors on ${\R}^n$. The dimension of $S^{m}\Rb^{n}$ is ${n+m-1}\choose m$.
	In particular, $S^0\Rb^{n}=\C$ and $S^1\Rb^{n}={\C}^n$.
	Let
	$\Sc(\Rb^{n}; S^{m}\Rb^{n})$ be the Schwartz space of $S^{m}\Rb^{n}$-valued functions that are called {\it rank $m$ smooth fast decaying symmetric tensor fields} on ${\R}^n$. The {\it ray transform} is the linear bounded operator
	\begin{equation}
	I: \Sc(\Rb^{n}; S^{m}\Rb^{n})\to \Sc (T\Sb^{n-1})
	\label{1.1}
	\end{equation}
	that is defined, for $f=(f_{i_1\dots i_m})\in\Sc(\Rb^{n}; S^{m}\Rb^{n})$, by
	\begin{equation}
	If (x,\xi)=\int\limits_{-\infty}^\infty f_{i_1\dots i_m}(x+t\xi)\,\xi^{i_1}\dots\xi^{i_m}\,\md t=\int\limits_{-\infty}^\infty \l f(x+t\xi),\xi^m\r\,\md t\quad\big((x,\xi)\in T{\S}^{n-1}\big).
	\label{1.2}
	\end{equation}
	We use the Einstein summation rule: the summation from 1 to $n$ is assumed over every index repeated in lower and upper positions in a monomial.
	To adopt our formulas to the Einstein summation rule, we use either lower or upper indices for denoting coordinates of vectors and tensors. For instance, $\xi^i=\xi_i$ in \eqref{1.2}. There is no difference between covariant and contravariant tensors since we use Cartesian coordinates only.
	Being initially defined by \eqref{1.2} on smooth fast decaying tensor fields, the operator \eqref{1.1} then extends to some wider spaces of tensor fields.

	The study of the ray transform is motivated by several applications.
	In the case of $m=0$ (when $f$ is a function), the ray transform is the main mathematical tool of Computer Tomography.
	In the case of $m=1$ (when $f$ is a vector field), the operator $I$ is called the {\it Doppler transform} and serves as the main mathematical tool of Doppler Tomography. In the cases of $m=2$ and of $m=4$, the operator $I$ and some its relatives are applied to various problems of tomography of anisotropic media, see \cite[Chapters 6,7]{mb} and \cite{LS,SW}.

	The operator $I$ has a big null-space in the case of $m>0$. A symmetric tensor field can be uniquely decomposed into its solenoidal and potential parts \cite[Theorem 2.6.2]{mb}, and the potential part lies in the null-space. Given $If$, one can recover only the solenoidal part of $f$, and there is a reconstruction formula \cite[Theorem 2.12.2]{mb}. This naturally leads to the question of what additional information should be added to the data $If$ for the unique recovery of the entire tensor field $f$. One possibility is to consider the {\it momentum ray transforms}
	$$
	I^k:{\mathcal S}({\R}^n;S^m{\R}^n)\rightarrow{\mathcal S}(T{\S}^{n-1})
	$$
	that are defined for $k=0,1,\dots$ as follows:
	\begin{equation}
	(I^k\!f)(x,\xi)=\int\limits_{-\infty}^\infty t^k\l f(x+t\xi),\xi^m\r\,\md t\quad\big((x,\xi)\in T{\S}^{n-1}\big).
	\label{1.3}
	\end{equation}
	In particular, $I^0=I$. A rank $m$ symmetric tensor field $f$ is uniquely determined by the functions $(I^0\!f,I^1\!f,\dots, I^m\!f)$, see \cite[Theorem 2.17.2]{mb} and \cite{AM}.
	
	The momentum ray transforms are primary objects of study of this paper, and we have three goals:
	
	(1) To obtain an algorithm for recovering a rank $m$ symmetric tensor field $f$ from the data $(I^0\!f,I^1\!f,\dots, I^m\!f)$.
	
	(2) To derive a version of the Reshetnyak formula \cite{Sh3} that expresses the norm $\|f\|_{H^s_t}$ through some norms of the functions $(I^0\!f,I^1\!f,\dots I^m\!f)$. The $H^{s}_{t}$-norm is a modification of the Sobolev norm weighted differently at high and low frequencies, see Section 2 for the precise definition.
	
	(3) To obtain stability estimates in terms of $H^s_t$-norms.
	
	The first goal is achieved for arbitrary $m$ in any dimension $n\ge2$. The Reshetnyak formula and stability estimate are obtained in the cases of $m=1$ and $m=2$ only. We believe our approach works for any $m$, but the bulkiness of the  Reshetnyak formula grows very fast with $m$.

	The paper is organized as follows.
	In Section \ref{Prelim}, we discuss basic properties of momentum ray transforms and state a few preliminaries. Section \ref{Inversion} presents the inversion algorithm. Section \ref{Reshetnyak} is devoted to the Reshetnyak formula and stability estimates. Finally, in Section \ref{2D}, we restrict ourselves to the 2-dimensional case and propose an alternate approach based on the fact that there are natural coordinates on $T{\S}^1={\S^1}\times{\R}$.

	\section{Preliminaries}\label{Prelim}
	
	\subsection{Basic properties of momentum ray transforms}
	
	First of all we observe that the right-hand side of \eqref{1.3} makes sense for all $(x,\xi)\in{\R}^n\times({\R}^n\setminus\{0\})$. We define the continuous linear operators
	\begin{equation}
	J^k:{\mathcal S}({\R}^n;S^m{\R}^n)\rightarrow C^\infty(T{\S}^{n-1})\quad(k=0,1,2,\dots)
	\label{2.1}
	\end{equation}
	by
	\begin{equation}
	(J^kf)(x,\xi)=\int\limits_{-\infty}^\infty t^k\l f(x+t\xi),\xi^m\r\,\md t\quad\mbox{for}\quad(x,\xi)\in{\R}^n\times({\R}^n\setminus\{0\}).
	\label{2.2}
	\end{equation}
	The data $(I^0\!f,I^1\!f,\dots, I^m\!f)$ and $(J^0\!f,J^1\!f,\dots, J^m\!f)$ are equivalent as we will demonstrate right now. Therefore the operators \eqref{2.1} are also called {\it momentum ray transforms}. The function $J^k\!f$ is sometimes more convenient than $I^k\!f$ because the partial derivatives $\frac{\partial(J^k\!f)}{\partial x^i}$ and $\frac{\partial(J^k\!f)}{\partial\xi^i}$ are well defined. On the other hand, the function $(I^k\!f)(x,\xi)$ obeys good decay conditions in the second argument.
	
	For a tensor field $f\in{\mathcal S}({\R}^n;S^m{\R}^n)$, the function $(J^k\!f)(x,\xi)$ possesses the following homogeneity in the second argument
	\begin{equation}
	(J^k\!f)(x,t\xi)=\frac{t^{m-k}}{|t|}(J^k\!f)(x,\xi)\quad(0\neq t\in{\R})
	\label{2.3}
	\end{equation}
	and has the following property in the first argument
	\begin{equation}
	(J^k\!f)(x+t\xi,\xi)=\sum\limits_{\ell=0}^k{k\choose\ell}(-t)^{k-\ell}(J^\ell\!f)(x,\xi)\quad(t\in{\R}).
	\label{2.4}
	\end{equation}
	Compare with \cite[Formulas (2.1.11)--(2.1.12)]{mb}. These properties easily follow from the definition \eqref{2.2}.
	
	Comparing \eqref{1.3} and \eqref{2.2}, we see that
	\begin{equation}
	I^k\!f=J^k\!f|_{T{\S}^{n-1}}.
	\label{2.5}
	\end{equation}
	On the other hand,
	\begin{equation}
	(J^k\!f)(x,\xi)=|\xi|^{m-2k-1}\sum\limits_{\ell=0}^k(-1)^{k-\ell}{k\choose\ell}|\xi|^\ell\l\xi,x\r^{k-\ell}\,(I^\ell\!f)
	\Big(x-\frac{\l\xi,x\r}{|\xi|^2}\xi,\frac{\xi}{|\xi|}\Big)
	\label{2.6}
	\end{equation}
	as easily follows from \eqref{2.3}--\eqref{2.4}. Compare with \cite[Formula (2.1.13)]{mb}.
	
	Formulas \eqref{2.4} and \eqref{2.6} mean in particular that the operator $I^k$ must always be  considered together with lower order momenta $(I^0,\dots,I^{k-1})$, i.e., the data $(I^0f,\dots,I^kf)$ must always be used instead of $I^kf$.
	
	There are two important first order differential operators on symmetric tensor fields: the {\it inner derivative} operator
	$\md:C^\infty({\R}^n;S^{m}{\R}^n)\to C^\infty({\R}^n;S^{m+1}{\R}^n)$ and the {\it divergence} operator $\delta:C^\infty({\R}^n;S^{m}{\R}^n)\to C^\infty({\R}^n;S^{m-1}{\R}^n)$, see the definitions in \cite{mb,Sh3}.
	Operators  $I^k$ are related to the inner derivative by the formula
	$$
	I^k(df)=-kI^{k-1}\!f
	$$
	which is valid at least for $f\in {\mathcal S}({\R}^n;S^m{\R}^n)$. This easily follows from \eqref{2.2} with the help of integration by parts. In particular, $I^k(d^\ell f)=0$ for $k<\ell$.
	
	Let us also mention the transformation law for operators $I^k$ under a change of the origin in ${\R}^n$. Given $a\in{\R}^n$, set $f_a(x)=f(x+a)$. As easily follows from \eqref{2.5} and \eqref{2.6},
	\begin{equation}
	(I^k\!f_a)(x,\xi)=\sum\limits_{\ell=0}^k(-1)^{k-\ell}{k\choose\ell}\l a,\xi\r^{k-\ell}\,(I^\ell\! f)(x+a-\l a,\xi\r\xi,\xi)\quad\mbox{for}\quad
	(x,\xi)\in T{\S}^{n-1}.
	\label{2.7}
	\end{equation}
	We are going to derive a formula that expresses $\|f\|_{H^s_t}$ through $(I^0f,I^1f,\dots I^mf)$. Since $\|f\|_{H^s_t}=\|f_a\|_{H^s_t}$, the expression must be invariant under the transformation \eqref{2.7}.

	\subsection{Momentum ray transforms and the Fourier transform}
	
	We use the Fourier transform $F:{\mathcal S}({\R}^n)\rightarrow{\mathcal S}({\R}^n),\ f\mapsto\wh f$ in the following form (hereafter $\textsl{i}$ is the imaginary unit and $y$ is the Fourier dual variable of $x$):
	$$
	\wh{f}(y)=\frac{1}{(2\pi)^{n/2}}\int\limits_{{\R}^n}
	e^{-\I\langle y,x\rangle} f(x)\, \md x.
	$$
	The Fourier transform $F:{\mathcal S}({\R}^n;S^m{\R}^n)\rightarrow{\mathcal S}({\R}^n;S^m{\R}^n),\ f\mapsto\wh f$ on symmetric tensor fields is defined componentwise, i.e., ${\wh f}_{i_1\dots i_m}=\wh{f_{i_1\dots i_m}}$ (we use Cartesian coordinates only). Introduce also the
	Fourier transform $F:\Sc(T\Sb^{n-1})\rightarrow\Sc(T\Sb^{n-1}),\ \varphi\mapsto\wh\varphi$ on $T\Sb^{n-1}$ by
	\begin{equation}
	\wh\varphi(y,\xi)= \frac{1}{(2\pi)^{(n-1)/2}}\int\limits_{\xi^{\perp}}
	e^{-\I\langle y,x\rangle} \varphi(x,\xi)\, \md x,
	\label{2.8}
	\end{equation}
	where $\md x$ is the $(n-1)$-dimensional Lebesgue measure on the hyperplane $\xi^{\perp}=\{x\in{\R}^n\mid\l\xi,x\r=0\}$. It is the standard Fourier transform in the $(n-1)$-dimensional variable $x$ while $\xi\in \Sb^{n-1}$ is considered as a parameter.
	
	The Fourier transform of the momentum ray transform is given by the formula
	\begin{equation}
	\widehat{I^k\!f}(y,\xi)=(2\pi)^{1/2}F_{x\rightarrow y}[\l\xi,x\r^k \l f(x),\xi^m\r]
	=(2\pi)^{1/2}{\I}^k\,\l d^k\wh{f}(y),\xi^{m+k}\r
	\label{2.9}
	\end{equation}
	for $(y,\xi)\in T{\S}^{n-1}$.
	Indeed,
	$$
	\begin{aligned}
	\widehat{I^k\!f}(y,\xi)&=(2\pi)^{(1-n)/2}\int\limits_{\xi^\bot}e^{-\textsl{i}\l y,x\r}(I^k\!f)(x,\xi)\,\md x\\
	&=(2\pi)^{(1-n)/2}\int\limits_{\xi^\bot}\int\limits_{-\infty}^\infty e^{-\textsl{i}\l y,x\r}t^k\l f(x+t\xi),\xi^m\r\,\md t\md x.
	\end{aligned}
	$$
	On assuming $(y,\xi)\in T{\S}^{n-1}$, change the integration variables as $z=x+t\xi$
	$$
	\widehat{I^k\!f}(y,\xi)
	=(2\pi)^{1/2}\big\langle(2\pi)^{-n/2}\int\limits_{{\R}^n} e^{-\textsl{i}\l y,z\r}\l\xi,z\r^k f(z)\,\md z,\xi^m\big\rangle.
	$$
	On using the equality
	$$
	(2\pi)^{-n/2}\int\limits_{{\R}^n} e^{-\textsl{i}\l y,z\r}\l\xi,z\r^k f(z)\,dz={\textsl{i}}^k\l\xi,\partial_y\rangle^k\hat f(y),
	$$
	we obtain
	$$
	\widehat{I^k\!f}(y,\xi)
	=(2\pi)^{1/2}{\textsl{i}}^k\l\xi,\partial_y\rangle^k\l\hat f(y),\xi^m\r.
	$$
	This coincides with \eqref{2.9}.
	
	\subsection{The Reshetnyak formula for scalar functions}
	
	Recall \cite{Sh3} that the Hilbert space $H^s_t({\R}^n)$ is defined for $s\in\R$ and $t>-n/2$ as the completion of ${\mathcal S}({\R}^n)$ with respect to  the norm
	\begin{equation}
	\|f\|^2_{H^s_t({\R}^n)}=\int\limits_{{\R}^n}|y|^{2t}(1+|y|^2)^{s-t}|\wh f(y)|^2\,\md y.
	\label{2.10}
	\end{equation}
	Similarly, the Hilbert space $H^s_t(T{\S}^{n-1})$ is defined for $s\in\R$ and $t>-(n-1)/2$ as the completion of ${\mathcal S}(T{\S}^{n-1})$ with respect to  the norm
	\begin{equation}
	\|\varphi\|^2_{H^s_t(T{\S}^{n-1})}=\frac{1}{2\pi}\int\limits_{{\S}^{n-1}}\int\limits_{\xi^\bot}|y|^{2t}(1+|y|^2)^{s-t}|\wh\varphi(y,\xi)|^2\,\md y\md \xi,
	\label{2.11}
	\end{equation}
	where $\md\xi$ is the volume form on the sphere ${\S}^{n-1}$ induced by the Euclidean metric of ${\R}^n$.
	
	The following statement is a particular case of $m=0$ in \cite[Theorem 4.2]{Sh3}. Given a function $f\in{\mathcal S}({\R}^n)$, the
	{\it Reshetnyak formula}
	\begin{equation}
	\|f\|^2_{H^s_t({\R}^n)}=a_n\|If\|^2_{H^{s+1/2}_{t+1/2}(T{\S}^{n-1})}
	\label{2.12}
	\end{equation}
	holds for every $s\in\R$ and every $t>-n/2$, where
	\begin{equation}
	a_n=\frac{\Gamma\big((n-1)/2\big)}{2\pi^{(n-1)/2}}.
	\label{2.13}
	\end{equation}
	(Unfortunately, there is a misprint in the formula for the coefficient $a_k=a_k(m,n)$ in Theorem 4.2 of \cite{Sh3}. The right formula is presented in \cite[Formula (2.15.3)]{mb}, the coefficients are independent of $(s,t)$.) Formula \eqref{2.12} will be used throughout the paper.
	
	\section{Inversion algorithm for momentum ray transforms}\label{Inversion}
	
	Let us recall the definition of the {\it partial symmetrization}
	$$
	\sigma(i_1\dots i_r)u_{i_1\dots i_rj_1\dots j_s}=\frac{1}{r!}\sum\limits_{\pi\in\Pi_r}u_{i_{\pi(1)}\dots i_{\pi(r)}j_1\dots j_s},
	$$
	where the summation is performed over the set $\Pi_r$ of all permutations of the set $\{1,\dots r\}$.
	
	The following statement is the main ingredient of our algorithm for recovering a rank $m$ symmetric tensor field $f$ from the data $(I^0\!f,\dots, I^m\!f)$.
	
	\begin{theorem} \label{Th3.1}
		Given a tensor field $f=(f_{i_1\dots i_m})\in{\mathcal S}({\R}^n;S^m{\R}^n)$, equalities
		\begin{equation}
		Jf_{i_1\dots i_m}=\frac{1}{m!}\sigma(i_1\dots i_m)\sum\limits_{k=0}^m(-1)^k{m\choose k}\frac{\partial^m(J^k\!f)}
		{\partial x^{i_1}\dots\partial x^{i_k}\partial\xi^{i_{k+1}}\dots\partial\xi^{i_m}}
		\label{3.1}
		\end{equation}
		hold for all indices $(i_1,\dots,i_m)$,
		where the left-hand side is the  result of applying the ray transform $J=J^0$ to the coordinate $f_{i_1\dots i_m}$ considered as a scalar function on ${\R}^n$ (we use Cartesian coordinates only).
	\end{theorem}
	
	\begin{proof}
		Equality \eqref{3.1} trivially holds for $m=0$. We proceed by induction on $m$. Assume \eqref{3.1} to be valid for tensor fields
		$f\in{\mathcal S}({\R}^n;S^m{\R}^n)$ with some $m$.
		
		Now assume $f\in{\mathcal S}({\R}^n;S^{m+1}{\R}^n)$. Definition \eqref{2.2} can be written as
		$$
		(J^k\!f)(x,\xi)=\xi^{p_1}\dots\xi^{p_{m+1}}\int\limits_{-\infty}^\infty t^k\, f_{p_1\dots p_{m+1}}(x+t\xi)\,\md t.
		$$
		Differentiate this equality with respect to $\xi^{i_{m+1}}$
		\begin{equation}
		\begin{aligned}
		\frac{\partial(J^k\!f)}{\partial\xi^{i_{m+1}}}&=(m+1)\xi^{p_1}\dots\xi^{p_m}\int\limits_{-\infty}^\infty t^k\, f_{p_1\dots p_m i_{m+1}}(x+t\xi)\,\md t\\
		&+\frac{\partial}{\partial x^{i_{m+1}}}\Big(\xi^{p_1}\dots\xi^{p_{m+1}}\int\limits_{-\infty}^\infty t^{k+1}\, f_{p_1\dots p_{m+1}}(x+t\xi)\,\md t\Big).
		\end{aligned}
		\label{3.2}
		\end{equation}
		
		Let us fix a value of the index $i_{m+1}$ and define the tensor field $\tilde f\in{\mathcal S}({\R}^n;S^m{\R}^n)$ by
		$$
		{\tilde f}_{i_1\dots i_m}=f_{i_1\dots i_m i_{m+1}}.
		$$
		Then \eqref{3.2} can be written as
		$$
		\frac{\partial(J^k\!f)}{\partial\xi^{i_{m+1}}}=(m+1)J^k\!\tilde f+\frac{\partial(J^{k+1}\!f)}{\partial x^{i_{m+1}}}.
		$$
		From this,
		$$
		J^k\!\tilde f=\frac{1}{m+1}\Big(\frac{\partial(J^k\!f)}{\partial\xi^{i_{m+1}}}-\frac{\partial(J^{k+1}\!f)}{\partial x^{i_{m+1}}}\Big).
		$$
		Differentiate this equality to obtain
		\begin{equation}
		\begin{aligned}
		&\frac{\partial^m(J^k\!\tilde f)}{\partial x^{i_1}\dots\partial x^{i_k}\partial\xi^{i_{k+1}}\dots\partial\xi^{i_m}}=\\
		&=\frac{1}{m+1}\Big(\frac{\partial^{m+1}(J^k\!f)}{\partial x^{i_1}\dots\partial x^{i_k}\partial\xi^{i_{k+1}}\dots\partial\xi^{i_{m+1}}}
		-\frac{\partial^{m+1}(J^{k+1}\!f)}{\partial x^{i_1}\dots\partial x^{i_k}\partial x^{i_{m+1}}\partial \xi^{i_{k+1}}\dots\partial x^{i_m}}\Big).
		\end{aligned}
		\label{3.3}
		\end{equation}
		
		By the induction hypothesis,
		$$
		Jf_{i_1\dots i_{m+1}}=J{\tilde f}_{i_1\dots i_m}=\frac{1}{m!}\sigma(i_1\dots i_m)\sum\limits_{k=0}^m(-1)^k{m\choose k}\frac{\partial^m(J^k\!{\tilde f})}
		{\partial x^{i_1}\dots\partial x^{i_k}\partial\xi^{i_{k+1}}\dots\partial\xi^{i_m}}.
		$$
		Substitute value \eqref{3.3} into the last formula
		$$
		\begin{aligned}
		Jf_{i_1\dots i_{m+1}}&=\frac{1}{(m+1)!}\sigma(i_1\dots i_{m+1})\sum\limits_{k=0}^m(-1)^k{m\choose k}
		\frac{\partial^{m+1}(J^k\!f)}{\partial x^{i_1}\dots\partial x^{i_k}\partial\xi^{i_{k+1}}\dots\partial\xi^{i_{m+1}}}\\
		&-\frac{1}{(m+1)!}\sigma(i_1\dots i_{m+1})\sum\limits_{k=0}^m(-1)^k{m\choose k}
		\frac{\partial^{m+1}(J^{k+1}\!f)}{\partial x^{i_1}\dots\partial x^{i_k}\partial x^{i_{m+1}}\partial \xi^{i_{k+1}}\dots\partial x^{i_m}}.
		\end{aligned}
		$$
		We have replaced the symmetrization $\sigma(i_1\dots i_m)$ by the stronger operator $\sigma(i_1\dots i_{m+1})$ because the left-hand side is symmetric in the indices $(i_1,\dots, i_{m+1})$. In the second sum on the right-hand side, we change the summation index as $k=k'-1$. After the change, we again use the notation $k$ instead of $k'$. In such the way, we transform the formula to the form
		$$
		\begin{aligned}
		Jf_{i_1\dots i_{m+1}}&=\frac{1}{(m+1)!}\sigma(i_1\dots i_{m+1})\sum\limits_{k=0}^m(-1)^k{m\choose k}
		\frac{\partial^{m+1}(J^k\!f)}{\partial x^{i_1}\dots\partial x^{i_k}\partial\xi^{i_{k+1}}\dots\partial\xi^{i_{m+1}}}\\
		&+\frac{1}{(m+1)!}\sigma(i_1\dots i_{m+1})\sum\limits_{k=1}^{m+1}(-1)^k{m\choose {k-1}}
		\frac{\partial^{m+1}(J^k\!f)}{\partial x^{i_1}\dots\partial x^{i_{k-1}}\partial x^{i_{m+1}}\partial \xi^{i_k}\dots\partial x^{i_m}}.
		\end{aligned}
		$$
		We assume binomial coefficients ${m\choose k}=\frac{m!}{k!(m-k)!}$ to be defined for all integers $m$ and $k$ under the agreement: ${m\choose k}=0$ if either $m<0$ or $k<0$ or $k>m$. Therefore both summations can be extended to the limits $0\le k\le m+1$. Besides this, we can write indices $(i_1,\dots, i_{m+1})$ in an arbitrary order on the right-hand side because of the presence of the symmetrization $\sigma(i_1\dots i_{m+1})$.
		With the help of the Pascal relation
		$ {m\choose k}+{m\choose{k-1}}={{m+1}\choose k}$, the formula takes the form
		$$
		Jf_{i_1\dots i_{m+1}}=\frac{1}{(m+1)!}\sigma(i_1\dots i_{m+1})\sum\limits_{k=0}^{m+1}(-1)^k{{m+1}\choose k}
		\frac{\partial^{m+1}(J^k\!f)}{\partial x^{i_1}\dots\partial x^{i_k}\partial\xi^{i_{k+1}}\dots\partial\xi^{i_{m+1}}}.
		$$
		This finishes the induction step.
	\end{proof}
	
	Let us recall the inversion formula for recovering a scalar function $f\in{\mathcal S}({\R}^n)$ from $If$:
	\begin{equation}
	f(x)=\frac{\Gamma\big((n-1)/2\big)}{4\pi^{(n+1)/2}}(-\Delta)^{1/2}\int\limits_{{\S}^{n-1}}(If)(x-\l\xi,x\r\xi,\xi)\,\md\xi.
	\label{3.4}
	\end{equation}
	This is a particular case of $m=0$ of \cite[Theorem 2.12.2]{mb}.
	
	Our algorithm for recovering a symmetric $m$-tensor field $f$ from the collection of functions $(I^0\!f,\dots, I^m\!f)$ is as follows. Given $(I^0\!f,\dots, I^m\!f)$,
	we find the data $(J^0\!f,\dots, J^m\!f)$ by formula \eqref{2.6}. Then we find the functions $Jf_{i_1\dots i_m}$ for all values of indices by \eqref{3.1} and find $If_{i_1\dots i_m}=Jf_{i_1\dots i_m}|_{T{\S}^{n-1}}$. From the latter data we recover all components $f_{i_1\dots i_m}$ of the field $f$ by formula \eqref{3.4}.
	
	The stability of the recovery procedure for a scalar function is completely described by the Reshetnyak formula \eqref{2.12}. For higher rank tensor fields, the stability question is more delicate because of the presence of $m^{\mathrm{th}}$ order derivatives in formula \eqref{3.1}. We will investigate the stability question in the next section.

	\section{Reshetnyak formula for momentum ray transforms}\label{Reshetnyak}
	
	In this section, we derive the Reshetnyak formula and stability estimate for $m=1$ and for $m=2$.
	
	\subsection{Operators $X_i$ and $\Xi_i$}
	
	In the cases of $m=1$ and of $m=2$, formula \eqref{3.1} takes the forms
	\begin{equation}
	Jf_i=\frac{\partial(J^0\!f)}{\partial\xi^i}-\frac{\partial(J^1\!f)}{\partial x^i}
	\label{4.1}
	\end{equation}
	and
	\begin{equation}
	Jf_{ij}=\frac{1}{2}\Big(\frac{\partial^2(J^0\!f)}{\partial \xi^i\partial \xi^j}
	-\frac{\partial^2(J^1\!f)}{\partial x^i\partial \xi^j}
	-\frac{\partial^2(J^1\!f)}{\partial x^j\partial \xi^j}
	+\frac{\partial^2(J^2\!f)}{\partial x^i\partial x^j}\Big)
	\label{4.2}
	\end{equation}
	respectively. We are going to rewrite these formulas in intrinsic terms of the manifold $T{\S}^{n-1}$.
	
	Introduce the vector fields on ${\R}^n\times({\R}^n\setminus\{0\})=\{(x,\xi)\in {\R}^n\times{\R}^n\mid\xi\neq0\}$
	\begin{equation}
	\begin{aligned}
	{\tilde X}_i&=\frac{\partial}{\partial x^i}-\xi_i\xi^p\frac{\partial}{\partial x^p}\quad(1\le i\le n),\\
	{\tilde \Xi}_i&=\frac{\partial}{\partial \xi^i}-x_i\xi^p\frac{\partial}{\partial x^p}-\xi_i\xi^p\frac{\partial}{\partial \xi^p}\quad(1\le i\le n).
	\end{aligned}
	\label{4.3}
	\end{equation}
	The notations ${\tilde X}_i$ and ${\tilde \Xi}_i$ are chosen because the derivatives $\frac{\partial}{\partial x^i}$ and $\frac{\partial}{\partial \xi^i}$ are in some sense leading terms on the right-hand sides of \eqref{4.3}.
	
	\begin{lemma} \label{L8.1}
		At every point $(x,\xi)\in T{\S}^{n-1}$, vectors ${\tilde X}_i(x,\xi)$ and ${\tilde \Xi}_i(x,\xi)\ (1\le i\le n)$ are tangent to $T{\S}^{n-1}$. Let $X_i$ and $\Xi_i$ be the restrictions of vector fields ${\tilde X}_i$ and ${\tilde \Xi}_i$ to the manifold $T{\S}^{n-1}$ respectively. Thus, $X_i$ and $\Xi_i$ are smooth vector fields on $T{\S}^{n-1}$ and can be considered as first order differential operators
		$$
		X_i,\Xi_i:C^\infty(T{\S}^{n-1})\rightarrow C^\infty(T{\S}^{n-1}).
		$$
		The operators satisfy
		\begin{equation}
		[X_i,X_j]=0,
		\label{4.4}
		\end{equation}
		\begin{equation}
		[\Xi_i,\Xi_j]=x_iX_j-x_jX_i+\xi_i\Xi_j-\xi_j\Xi_i,
		\label{4.5}
		\end{equation}
		\begin{equation}
		[X_i,\Xi_j]=\xi_iX_j.
		\label{4.6}
		\end{equation}
		At every point $(x,\xi)\in T{\S}^{n-1}$, vectors $X_i(x,\xi),\Xi_i(x,\xi)\ (1\le i\le n)$ generate the tangent space $T_{(x,\xi)}(T{\S}^{n-1})$ and satisfy
		\begin{equation}
		\xi^iX_i(x,\xi)=0,\quad \xi^i\Xi_i(x,\xi)=0.
		\label{4.7}
		\end{equation}
	\end{lemma}
	
	\begin{proof}
		Definition \eqref{4.3} implies
		$$
		{\tilde X}_i|\xi|^2=0,\quad{\tilde \Xi}_i|\xi|^2=2\xi_i(1-|\xi|^2),\quad
		{\tilde X}_i\l\xi,x\r=\xi_i(1-|\xi|^2),\quad {\tilde \Xi}_i\l\xi,x\r=x_i(1-|\xi|^2)-\xi_i\l\xi,x\r.
		$$
		Right-hand sides of these equalities vanish on $T{\S}^{n-1}$. This proves the first statement.
		
		From definition \eqref{4.3},
		$$
		\xi^i{\tilde X}_i=(1-|\xi|^2)\xi^p\frac{\partial}{\partial x^p},\quad
		\xi^i{\tilde \Xi}_i=\l\xi,x\r\xi^p\frac{\partial}{\partial x^p}-(1-|\xi|^2)\xi^p\frac{\partial}{\partial \xi^p}.
		$$
		Right-hand sides of these equalities vanish on $T{\S}^{n-1}$. This proves \eqref{4.7}.
		
		Recall the well known formula:
		$$
		[fX,gY]=fg[X,Y]+f(Xg)Y-g(Yf)X
		$$
		for vector fields $X,Y$ and for functions $f,g$. On using this formula, one easily derives from definition \eqref{4.3}
		\begin{equation}
		[{\tilde X}_i,{\tilde \Xi}_j]=\xi_i\frac{\partial}{\partial x^j}-\xi_i\xi_j\xi^p\frac{\partial}{\partial x^p}.
		\label{4.8}
		\end{equation}
		On the other hand,
		\begin{equation}
		\xi_i{\tilde X}_j=\xi_i\Big(\frac{\partial}{\partial x^j}-\xi_i\xi^p\frac{\partial}{\partial x^p}\Big)
		=\xi_i\frac{\partial}{\partial x^j}-\xi_i\xi_j\xi^p\frac{\partial}{\partial x^p}.
		\label{4.9}
		\end{equation}
		Comparing \eqref{4.8} and \eqref{4.9}, we see that $[{\tilde X}_i,{\tilde \Xi}_j]=\xi_i{\tilde X}_j$. This proves \eqref{4.6}. Formulas \eqref{4.4}--\eqref{4.5} are proved in a similar way.
		
		Finally we prove that, at a point $(x,\xi)\in T{\S}^{n-1}$, the vectors $X_i(x,\xi),\Xi_i(x,\xi)\ (1\le i\le n)$ generate  the tangent space $T_{(x,\xi)}(T{\S}^{n-1})$. To this end we have to demonstrate that any linear dependence between these vectors is actually a corollary of \eqref{4.7}.
		Assume that
		\begin{equation}
		\alpha^i X_i+\beta^i \Xi_i=0
		\label{4.10}
		\end{equation}
		with some coefficients $\alpha^i,\beta^i$. Substitute values \eqref{4.3} into this equality
		$$
		\alpha^i\Big(\frac{\partial}{\partial x^i}-\xi_i\xi^p\frac{\partial}{\partial x^p}\Big)
		+\beta^i\Big(\frac{\partial}{\partial \xi^i}-x_i\xi^p\frac{\partial}{\partial x^p}-\xi_i\xi^p\frac{\partial}{\partial \xi^p}\Big)=0.
		$$
		This can be written in the form
		$$
		\big(\alpha^i(\delta_i^p-\xi_i\xi^p)-\beta^ix_i\xi^p)\frac{\partial}{\partial x^p}
		+\beta^i(\delta_i^p-\xi_i\xi^p)\frac{\partial}{\partial \xi^p}=0,
		$$
		where $\delta^i_j$ is the Kronecker tensor.
		Since vectors $\frac{\partial}{\partial x^p},\frac{\partial}{\partial \xi^p}\ (p=1,\dots,n)$ are linearly independent, the last equation implies
		\begin{equation}
		(\delta_i^p-\xi_i\xi^p)\alpha^i-x_i\xi^p\beta^i=0\quad(p=1,\dots,n),
		\label{4.11}
		\end{equation}
		\begin{equation}
		(\delta_i^p-\xi_i\xi^p)\beta^i=0\quad(p=1,\dots,n).
		\label{4.12}
		\end{equation}
		
		At a point $(x,\xi)\in T{\S}^{n-1}$, the rank of the matrix $(\delta_i^p-\xi_i\xi^p)$ of system \eqref{4.12} is equal to $n-1$ and any solution to the system is of the form $\beta^i=\beta_0\xi^i$. System \eqref{4.11} takes now the form
		$$
		(\delta_i^p-\xi_i\xi^p)\alpha^i=0\quad(p=1,\dots,n).
		$$
		As before, this implies $\alpha^i=\alpha_0\xi^i$. Thus, equation \eqref{4.10} is actually of the form
		$$
		\alpha_0\xi^iX_i+\beta_0\xi^i\Xi_i=0.
		$$
		This means that any linear dependence between vectors $X_i(x,\xi),\Xi_i(x,\xi)\ (1\le i\le n)$ is actually a corollary of \eqref{4.7}.
	\end{proof}
	
	We can now present the intrinsic forms of equations \eqref{4.1} and \eqref{4.2}.
	
	\begin{theorem} \label{Th8.1}
		For a vector field $f=(f_i)\in{\mathcal S}({\R}^n;{\C}^n)$, the equality
		\begin{equation}
		If_i=(\Xi_i+\xi_i)(I^0\!f)-X_i(I^1\!f)
		\label{4.13}
		\end{equation}
		holds on $T{\S}^{n-1}$ for every $i$, where the left-hand side is the result of applying the ray transform $I=I^0$ to the coordinates $f_i$ considered as scalar functions on ${\R}^n$.
	\end{theorem}
	
	\begin{theorem} \label{Th8.2}
		For a tensor field $f=(f_{ij})\in \Sc(\Rb^{n};S^{2}\Rb^{n})$, the equality
		\begin{equation}
		\begin{aligned}
		If_{ij}=\frac{1}{2}\sigma(ij)\Big[& X_i X_j(I^2\!f)
		-2\Xi_i X_j(I^1\!f)
		-4\xi_i X_j(I^1\!f)\\
		&+\Xi_i \Xi_j(I^0\!f)
		+x_i X_j(I^0\!f)
		+3\xi_i \Xi_j(I^0\!f)
		+3\delta_{ij}(I^0\!f)-\xi_i\xi_j(I^0\!f)\Big]
		\end{aligned}
		\label{4.14}
		\end{equation}
		holds  on $T{\S}^{n-1}$ for all indices $(i,j)$, where $\delta_{ij}$ is the Kronecker tensor and the left-hand side is the result of applying the ray transform $I=I^0$ to the coordinates $f_{ij}$ considered as scalar functions on ${\R}^n$.
	\end{theorem}
	
	We present the proof of Theorem \ref{Th8.1}. Theorem \ref{Th8.2} is proved in the same way although all involved calculations are more cumbersome.
	
	\begin{proof}[Proof of Theorem \ref{Th8.1}]
		
		By the very definition of $X_i$ and $\Xi_i$, the equalities
		$$
		X_i\varphi={\tilde X}_i\psi,\quad \Xi_i\varphi={\tilde \Xi}_i\psi
		$$
		hold on $T{\S}^{n-1}$ for an arbitrary function $\varphi\in C^\infty(T{\S}^{n-1})$, where $\psi$ is an arbitrary smooth extension of $\varphi$ to some neighborhood of $T{\S}^{n-1}$ in ${\R}^n\times({\R}^n\setminus\{0\})$. For every $k$, the function $J^k\!f$ is an extension of $I^k\!f$. Therefore
		$$
		X_i(I^k\!f)={\tilde X}_i(J^k\!f),\quad \Xi_i(I^k\!f)={\tilde \Xi}_i(J^k\!f)\quad \mbox{on}\quad T{\S}^{n-1}.
		$$
		Substitute values \eqref{4.3} to obtain
		\begin{equation}
		\left.
		\begin{aligned}
		X_i(I^k\!f)&=\Big(\frac{\partial}{\partial x^i}-\xi_i\xi^p\frac{\partial}{\partial x^p}\Big)(J^k\!f),\\
		\Xi_i(I^k\!f)&=\Big(\frac{\partial}{\partial \xi^i}-x_i\xi^p\frac{\partial}{\partial x^p}-\xi_i\xi^p\frac{\partial}{\partial \xi^p}\Big)(J^k\!f)
		\end{aligned}
		\right\}\quad \mbox{on}\ T{\S}^{n-1};
		\label{4.15}
		\end{equation}
		
		Let us specify \eqref{4.15} for $k=0$.
		As is seen from \eqref{2.3}, the function $(J^0\!f)(x,\xi)$ is positively homogeneous of zero degree in the second argument. By the Euler equation for homogeneous functions,
		\begin{equation}
		\xi^p\frac{\partial(J^0\!f)}{\partial\xi^p}=0.
		\label{4.16}
		\end{equation}
		Besides this, $J^0\!f$ satisfies (see \eqref{2.4})
		$$
		(J^0\!f)(x+t\xi,\xi)=(J^0\!f)(x,\xi)\quad(t\in{\R}).
		$$
		Differentiating this identity with respect to $t$ and then setting $t=0$, we obtain
		\begin{equation}
		\xi^p\frac{\partial(J^0\!f)}{\partial x^p}=0.
		\label{4.17}
		\end{equation}
		In virtue of \eqref{4.16}--\eqref{4.17}, equalities \eqref{4.15} for $k=0$ are simplified to the following ones:
		\begin{equation}
		\frac{\partial(J^0\!f)}{\partial x^i}=X_i(I^0\!f),\quad
		\frac{\partial(J^0\!f)}{\partial \xi^i}=\Xi_i(I^0\!f)\quad \mbox{on}\ T{\S}^{n-1}.
		\label{4.18}
		\end{equation}
		
		Let us specify \eqref{4.15} for $k=1$.
		As is seen from \eqref{2.3}, the function $(J^1\!f)(x,\xi)$ is positively homogeneous of degree $-1$ in the second argument. By the Euler equation for homogeneous functions,
		\begin{equation}
		\xi^p\frac{\partial(J^1\!f)}{\partial\xi^p}=-J^1\!f.
		\label{4.19}
		\end{equation}
		Besides this, $J^1\!f$ satisfies (see \eqref{2.4})
		$$
		(J^1\!f)(x+t\xi,\xi)=-t(J^0\!f)(x+t\xi,\xi)+(J^1\!f)(x,\xi)\quad(t\in{\R}).
		$$
		Differentiating this identity with respect to $t$ and then setting $t=0$, we obtain
		\begin{equation}
		\xi^p\frac{\partial(J^1\!f)}{\partial x^p}=-(J^0\!f).
		\label{4.20}
		\end{equation}
		In virtue of \eqref{4.19}--\eqref{4.20}, equalities \eqref{4.15} for $k=1$ are simplified to the following ones:
		\begin{equation}
		\frac{\partial(J^1\!f)}{\partial x^i}=X_i(I^1\!f)-\xi_i(I^0\!f),\quad
		\frac{\partial(J^1\!f)}{\partial \xi^i}=(\Xi_i-\xi_i)(I^1\!f)-x_i(I^0\!f)\quad \mbox{on}\ T{\S}^{n-1}.
		\label{4.21}
		\end{equation}
		
		Inserting value \eqref{4.18} of $\frac{\partial(J^0\!f)}{\partial\xi^i}$ and value \eqref{4.21} of $\frac{\partial(J^1\!f)}{\partial x^i}$ into \eqref{4.1}, we arrive to \eqref{4.13}.
	\end{proof}
	
	Recall that $y_i$ is the Fourier dual variable of $x_i$.
	The formulas for commuting the Fourier transform with operators $X_i$ and $\Xi_i$ are described by the following
	
	\begin{lemma} \label{L8.2}
		The equalities
		\begin{equation}
		\widehat{X_i\varphi}=\textsl{i}\,y_i\,\hat\varphi,\quad
		\widehat{\Xi_i\varphi}=\Xi_i\hat\varphi,\quad
		\widehat{x_i\varphi}=\textsl{i}\,X_i\,\hat\varphi
		\label{4.22}
		\end{equation}
		hold for every function $\varphi\in{\mathcal S}(T{\S}^{n-1})$ and for every $i$, $1\leq i\leq n$.
	\end{lemma}
	
	\begin{proof}
		Given a function $\varphi\in{\mathcal S}(T{\S}^{n-1})$, define $\psi\in C^\infty\big({\R}^n\times({\R}^n\setminus\{0\})\big)$ by
		$$
		\psi(x,\xi)=\varphi\big(x-\frac{\l\xi,x\r}{|\xi|^2}\xi,\frac{\xi}{|\xi|}\big).
		$$
		The function satisfies the identities
		$$
		\psi(x,t\xi)=\psi(x,\xi)\ (0\neq t\in{\R}),\qquad\psi(x,x+t\xi)=\psi(x,\xi)\ (t\in{\R})
		$$
		which imply (by the same arguments that were used for deriving \eqref{4.16}--\eqref{4.17})
		\begin{equation}
		\xi^p\frac{\partial\psi}{\partial x^p}=0,\quad \xi^p\frac{\partial\psi}{\partial \xi^p}=0.
		\label{4.23}
		\end{equation}
		
		For $(y,\xi)\in T{\S}^{n-1}$, the definition \eqref{2.8} of the Fourier transform can be written in terms of $\psi$:
		\begin{equation}
		\hat\varphi(y,\xi)=(2\pi)^{(1-n)/2}\int\limits_{\xi^\bot}e^{-\textsl{i}\,\l y,x\r}\psi(x,\xi)\,\md x.
		\label{4.24}
		\end{equation}
		This formula makes sense for any $(y,\xi)\in{\R}^n\times\lb{\R}^n\setminus\{0\}\rb$ sufficiently close to the submanifold $T{\S}^{n-1}\subset{\R}^n\times\lb{\R}^n\setminus\{0\}\rb$. Therefore the partial derivatives $\frac{\partial\hat\varphi}{\partial y^i}$ and $\frac{\partial\hat\varphi}{\partial \xi^i}$ make sense.
		
		By \eqref{4.3},
		$$
		X_i\varphi={\tilde X}_i\psi=\frac{\partial\psi}{\partial x^i}-\xi_i\xi^p\frac{\partial\psi}{\partial x^p}.
		$$
		With the help of \eqref{4.23}, this is simplified to
		$$
		X_i\varphi=\frac{\partial\psi}{\partial x^i}.
		$$
		Applying the Fourier transform to this equality, we obtain
		$$
		\widehat{X_i\varphi}(y,\xi)=(2\pi)^{(1-n)/2}\int\limits_{\xi^\bot}e^{-\textsl{i}\,\l y,x\r}\frac{\partial\psi}{\partial x^i}(x,\xi)\,\md x.
		$$
		From this with the help of integration by parts
		$$
		\widehat{X_i\varphi}(y,\xi)=\textsl{i}\,y_i\,(2\pi)^{-(n\!-\!1)/2}\int\limits_{\xi^\bot}e^{-\textsl{i}\,\l y,x\r}\psi(x,\xi)\,\md x
		=\textsl{i}\,y_i\,\hat\varphi(y,\xi).
		$$
		This proves the first of formulas \eqref{4.22}.
		
		Differentiate equality \eqref{4.24} with respect to $y^i$
		\begin{equation}
		\frac{\partial\hat\varphi}{\partial y^i}(y,\xi)=-\textsl{i}\,(2\pi)^{(1-n)/2}\int\limits_{\xi^\bot}e^{-\textsl{i}\,\l y,x\r}x_i\psi(x,\xi)\,\md x.
		\label{4.25}
		\end{equation}
		From this
		\begin{equation}
		\xi^p\frac{\partial\hat\varphi}{\partial y^p}(y,\xi)=-\textsl{i}\,(2\pi)^{(1-n)/2}\int\limits_{\xi^\bot}e^{-\textsl{i}\,\l y,x\r}\l\xi,x\r\psi(x,\xi)\,\md x=0.
		\label{4.26}
		\end{equation}
		
		By \eqref{4.3},
		$$
		\Xi_i\hat\varphi=\frac{\partial\hat\varphi}{\partial \xi^i}-y_i\xi^p\frac{\partial\hat\varphi}{\partial y^p}-\xi_i\xi^p\frac{\partial\hat\varphi}{\partial \xi^p}.
		$$
		In virtue of \eqref{4.26}, this is simplified to
		\begin{equation}
		\Xi_i\hat\varphi=\frac{\partial\hat\varphi}{\partial \xi^i}-\xi_i\xi^p\frac{\partial\hat\varphi}{\partial \xi^p}.
		\label{4.27}
		\end{equation}
		
		The differentiation of equality \eqref{4.24} with respect to $\xi^i$ is not very easy because the integration hyperplane $\xi^\bot$ depends on $\xi$. We first transform integral \eqref{4.24} into an integral over a hyperplane independent of $\xi$. Fix a vector $\xi_0\in{\S}^{n-1}$. For an arbitrary vector $\xi\in{\S}^{n-1}$ sufficiently close to $\xi_0$, the orthogonal projection
		\begin{equation}
		\xi_0^\bot\rightarrow\xi^\bot,\quad x'\mapsto x=x'-\l\xi,x'\r\xi
		\label{4.28}
		\end{equation}
		is one-to-one. We change the integration variable in \eqref{4.24} according to \eqref{4.28}. The Jacobian of the change is $\l\xi_0,\xi\r^{-1}$. After the change, formula \eqref{4.24} takes the form
		$$
		\hat\varphi(y,\xi)=\frac{(2\pi)^{(1-n)/2}}{\l\xi_0,\xi\r}\int\limits_{\xi_0^\bot}
		e^{-\textsl{i}\,\l y,x'-\l\xi,x'\r\xi\r}\,\psi(x'-\l\xi,x'\r\xi,\xi)\,\md x'.
		$$
		We can now differentiate this equality with respect to $\xi^i$
		$$
		\begin{aligned}
		\frac{\partial\hat\varphi}{\partial\xi^i}(y,\xi)&=-\frac{\xi_{0,i}}{\l\xi_0,\xi\r^2}(2\pi)^{(1-n)/2}\int\limits_{\xi_0^\bot}
		e^{-\textsl{i}\,\l y,x'-\l\xi,x'\r\xi\r}\,\psi(x'-\l\xi,x'\r\xi,\xi)\,\md x'\\
		&+\frac{1}{\l\xi_0,\xi\r}(2\pi)^{(1-n)/2}\int\limits_{\xi_0^\bot}
		\frac{\partial e^{-\textsl{i}\,\l y,x'-\l\xi,x'\r\xi\r}}{\partial\xi^i}\,\psi(x'-\l\xi,x'\r\xi,\xi)\,\md x'\\
		&+\frac{1}{\l\xi_0,\xi\r}(2\pi)^{(1-n)/2}\int\limits_{\xi_0^\bot}
		e^{-\textsl{i}\,\l y,x'-\l\xi,x'\r\xi\r}\frac{\partial\psi(x'-\l\xi,x'\r\xi,\xi)}{\partial\xi^i}\,\md x'.
		\end{aligned}
		$$
		Substituting the values
		$$
		\frac{\partial e^{-\textsl{i}\,\l y,x'-\l\xi,x'\r\xi\r}}{\partial\xi^i}
		=\textsl{i}\,e^{-\textsl{i}\,\l y,x'-\l\xi,x'\r\xi\r}(\l\xi,y\r x'_i+\l\xi,x'\r y_i)
		$$
		and
		$$
		\begin{aligned}
		\frac{\partial\psi(x'-\l\xi,x'\r\xi,\xi)}{\partial\xi^i}&=
		\frac{\partial\psi}{\partial\xi^i}(x'-\l\xi,x'\r\xi,\xi)
		-\l\xi,x'\r\frac{\partial\psi}{\partial x^i}(x'-\l\xi,x'\r\xi,\xi)\\
		&-x'_i\xi^p\frac{\partial\psi}{\partial x^p}(x'-\l\xi,x'\r\xi,\xi),
		\end{aligned}
		$$
		we obtain
		$$
		\begin{aligned}
		\frac{\partial\hat\varphi}{\partial\xi^i}(y,\xi)&=-\frac{\xi_{0,i}}{\l\xi_0,\xi\r^2}(2\pi)^{(1-n)/2}\int\limits_{\xi_0^\bot}
		e^{-\textsl{i}\,\l y,x'-\l\xi,x'\r\xi\r}\,\psi(x'-\l\xi,x'\r\xi,\xi)\,\md x'\\
		&+\frac{\textsl{i}}{\l\xi_0,\xi\r}(2\pi)^{(1-n)/2}\int\limits_{\xi_0^\bot}
		e^{-\textsl{i}\,\l y,x'-\l\xi,x'\r\xi\r}(\l\xi,y\r x'_i+\l\xi,x'\r y_i)\,\psi(x'-\l\xi,x'\r\xi,\xi)\,\md x'\\
		&+\frac{1}{\l\xi_0,\xi\r}(2\pi)^{(1-n)/2}\int\limits_{\xi_0^\bot}
		e^{-\textsl{i}\,\l y,x'-\l\xi,x'\r\xi\r}\frac{\partial\psi}{\partial\xi^i}(x'-\l\xi,x'\r\xi,\xi)\,\md x'\\
		&-\frac{1}{\l\xi_0,\xi\r}(2\pi)^{(1-n)/2}\int\limits_{\xi_0^\bot}
		e^{-\textsl{i}\,\l y,x'-\l\xi,x'\r\xi\r}\l\xi,x'\r\frac{\partial\psi}{\partial x^i}(x'-\l\xi,x'\r\xi,\xi)\,\md x'\\
		&-\frac{1}{\l\xi_0,\xi\r}(2\pi)^{(1-n)/2}\int\limits_{\xi_0^\bot}
		e^{-\textsl{i}\,\l y,x'-\l\xi,x'\r\xi\r}\,x'_i\xi^p\,\frac{\partial\psi}{\partial x^p}(x'-\l\xi,x'\r\xi,\xi)\,\md x'.
		\end{aligned}
		$$
		On assuming $(y,\xi_0)\in T{\S}^{n-1}$, we set $\xi=\xi_0$ in the latter formula. The formula simplifies to the following one:
		$$
		\begin{aligned}
		\frac{\partial\hat\varphi}{\partial\xi^i}(y,\xi_0)&=-\xi_{0,i}(2\pi)^{(1-n)/2}\int\limits_{\xi_0^\bot}
		e^{-\textsl{i}\,\l y,x'\r}\,\psi(x',\xi_0)\,\md x'\\
		&+(2\pi)^{(1-n)/2}\int\limits_{\xi_0^\bot}e^{-\textsl{i}\,\l y,x'\r}\,\frac{\partial\psi}{\partial\xi^i}(x',\xi_0)\,\md x'\\
		&-(2\pi)^{(1-n)/2}\int\limits_{\xi_0^\bot}
		e^{-\textsl{i}\,\l y,x'\r}\,x'_i\,\xi_0^p\,\frac{\partial\psi}{\partial x^p}(x',\xi_0)\,\md x'.
		\end{aligned}
		$$
		By \eqref{4.23}, the integrand of the last integral is identically equal to zero. Thus, replacing the notations $\xi_0$ and $x'$ with $\xi$ and $x$ respectively, we obtain
		\begin{equation}
		\frac{\partial\hat\varphi}{\partial\xi^i}(y,\xi)=
		(2\pi)^{(1-n)/2}\int\limits_{\xi^\bot}e^{-\textsl{i}\,\l y,x\r}\,\frac{\partial\psi}{\partial\xi^i}(x,\xi)\,\md x
		-\xi_i\,\hat\varphi(y,\xi)
		\quad\mbox{for}\quad(y,\xi)\in T{\S}^{n-1}.
		\label{4.29}
		\end{equation}
		
		With the help of \eqref{2.8}, we obtain from \eqref{4.29}
		\begin{equation}
		\xi^p\frac{\partial\hat\varphi}{\partial\xi^p}(y,\xi)=
		(2\pi)^{(1-n)/2}\int\limits_{\xi^\bot}e^{-\textsl{i}\,\l y,x\r}\,\xi^p\frac{\partial\psi}{\partial\xi^p}(x,\xi)\,\md x
		-\hat\varphi(y,\xi)=-\hat\varphi(y,\xi)
		\ \mbox{for}\ (y,\xi)\in T{\S}^{n-1}.
		\label{4.30}
		\end{equation}
		
		Substitute values \eqref{4.29} and \eqref{4.30} into \eqref{4.27} to obtain
		\begin{equation}
		(\Xi_i\hat\varphi)(y,\xi)=(2\pi)^{(1-n)/2}\int\limits_{\xi^\bot}e^{-\textsl{i}\,\l y,x\r}\,\frac{\partial\psi}{\partial\xi^i}(x,\xi)\,\md x
		\quad\mbox{for}\quad(y,\xi)\in T{\S}^{n-1}.
		\label{4.31}
		\end{equation}
		
		By \eqref{4.23},
		$$
		({\tilde \Xi}_i\psi)(x,\xi)=\frac{\partial\psi}{\partial\xi^i}(x,\xi)\quad\mbox{for}\quad(x,\xi)\in T{\S}^{n-1}.
		$$
		Therefore formula \eqref{4.31} can be written as
		$$
		(\Xi_i\hat\varphi)(y,\xi)=(2\pi)^{(1-n)/2}\int\limits_{\xi^\bot}e^{-\textsl{i}\,\l y,x\r}\,({\tilde \Xi}_i\psi)(x,\xi)\,\md x
		=\widehat{\Xi_i\varphi}(y,\xi)
		\quad\mbox{for}\quad(y,\xi)\in T{\S}^{n-1}.
		$$
		This proves the second of formulas \eqref{4.22}.
		
		By the definition of $X_i$,
		$$
		X_i\hat\varphi=\frac{\partial\hat\varphi}{\partial y^i}-\xi_i\xi^p\frac{\partial\hat\varphi}{\partial y^p}.
		$$
		In virtue of \eqref{4.26}, this is simplified to
		$$
		X_i\hat\varphi=\frac{\partial\hat\varphi}{\partial y^i}.
		$$
		With the help of \eqref{4.25}, this gives
		$$
		X_i\hat\varphi=-\textsl{i}\,(2\pi)^{(1-n)/2}\int\limits_{\xi^\bot}e^{-\textsl{i}\,\l y,x\r}x_i\varphi(x,\xi)\,\md x
		=-\textsl{i}\,\widehat{x_i\varphi}.
		$$
		This proves the last of formulas \eqref{4.22}.
	\end{proof}

	\subsection{Reshetnyak formula for vector fields}
	
	\begin{theorem} \label{Rf1}
		Given a vector field $f\in \Sc(\Rb^{n};{\C}^n)$, the equality
		\begin{equation}
		\begin{aligned}
		\|f\|^2_{H^s_t({\R}^n;{\C}^n)}=a_n\Big[&
		\sum\limits_{i=1}^n\|\Xi_i(I^0\!f)\|^2_{H^{s+1/2}_{t+1/2}(T{\S}^{n-1})}
		+\|I^0\!f\|^2_{H^{s+1/2}_{t+1/2}(T{\S}^{n-1})}\\
		&+\|I^1\!f\|^2_{H^{s+3/2}_{t+3/2}(T{\S}^{n-1})}
		+2\Re\big(\textsl{i}\,Z(I^0\!f),I^1\!f\big)_{H^{s+1}_{t+1}(T{\S}^{n-1})}\Big]
		\end{aligned}
		\label{4.32}
		\end{equation}
		holds for every $s\in\R$ and every $t>-n/2$, where the constant $a_n$ is defined by \eqref{2.13} and $Z$ is the first order pseudodifferential operator
		$$
		Z:{\mathcal S}(T{\S}^{n-1})\rightarrow C^\infty(T{\S}^{n-1})
		$$
		defined by
		\begin{equation}
		\widehat{Z\varphi}(y,\xi)=\frac{y^i}{|y|}(\Xi_i\hat\varphi)(y,\xi).
		\label{4.33}
		\end{equation}
	\end{theorem}
	
	\begin{proof}
		Apply the Fourier transform to formula \eqref{4.13}
		$$
		\widehat{If_i}=-\widehat{X_i(I^1\!f)}+\widehat{\Xi_i(I^0\!f)}+\xi_i\,\widehat{I^0\!f}.
		$$
		With the help of Lemma \ref{L8.2}, this gives
		\begin{equation}
		\widehat{If_i}=-\textsl{i}\,y_i\,\widehat{I^1\!f}+\Xi_i\widehat{I^0\!f}+\xi_i\,\widehat{I^0\!f}.
		\label{4.34}
		\end{equation}
		
		For a vector field $f=(f_i)$,
		$$
		\|f\|^2_{H^s_t({\R}^n;{\C}^n)}=\sum\limits_{i=1}^n\|f_i\|^2_{H^s_t({\R}^n)}.
		$$
		Applying the Reshetnyak formula \eqref{2.12} for scalar functions, we obtain
		$$
		\|f\|^2_{H^s_t({\R}^n;{\C}^n)}=a_n\sum\limits_{i=1}^n\|\widehat{If_i}\|^2_{H^{s+1/2}_{t+1/2}(T{\S}^{n-1})}.
		$$
		On using the definition \eqref{2.11} of the $H^{s+1/2}_{t+1/2}(T{\S}^{n-1})$-norm, this is written as
		\begin{equation}
		\|f\|^2_{H^s_t({\R}^n;{\C}^n)}=\frac{a_n}{2\pi}\int\limits_{{\S}^{n-1}}\int\limits_{\xi^\bot}|y|^{2t+1}(1+|y|^2)^{s-t}
		\sum\limits_{i=1}^n|\widehat{If_i}(y,\xi)|^2\,\md y\md \xi.
		\label{4.35}
		\end{equation}
		
		Now, we calculate the integrand on \eqref{4.35}. By \eqref{4.34},
		$$
		\sum\limits_{i=1}^n|\widehat{If_i}|^2=
		\sum\limits_{i=1}^n\Big(\Xi_i\widehat{I^0\!f}+\xi_i\,\widehat{I^0\!f}-\textsl{i}\,y_i\,\widehat{I^1\!f}\Big)\,
		\Big(\Xi_i\overline{\widehat{I^0\!f}}+\xi_i\,\overline{\widehat{I^0\!f}}+\textsl{i}\,y_i\overline{\widehat{I^1\!f}}\Big).
		$$
		After opening parentheses and using the equalities $|\xi|^2=1,\l\xi,y\r=0$, this gives
		$$
		\sum\limits_{i=1}^n|\widehat{If_i}|^2=
		\sum\limits_{i=1}^n|\Xi_i\widehat{I^0\!f}|^2+|\widehat{I^0\!f}|^2+|y|^2|\widehat{I^1\!f}|^2
		+2\Re\big((\xi^i\Xi_i)\widehat{I^0\!f}\cdot\overline{\widehat{I^0\!f}}\big)
		+2\Re\big(\textsl{i}\,(y^i\Xi_i)\widehat{I^0\!f}\cdot\overline{\widehat{I^1\!f}}\big).
		$$
		By \eqref{4.7}, $\xi^i\Xi_i=0$, and the formula takes the form
		$$
		\sum\limits_{i=1}^n|\widehat{If_i}|^2=
		\sum\limits_{i=1}^n|\Xi_i\widehat{I^0\!f}|^2+|\widehat{I^0\!f}|^2+|y|^2|\widehat{I^1\!f}|^2
		+2\Re\big(\textsl{i}\,(y^i\Xi_i)\widehat{I^0\!f}\cdot\overline{\widehat{I^1\!f}}\big).
		$$
		On using the operator \eqref{4.33}, this can be written as
		\begin{equation}
		\sum\limits_{i=1}^n|\widehat{If_i}|^2=
		\sum\limits_{i=1}^n|\Xi_i\widehat{I^0\!f}|^2+|\widehat{I^0\!f}|^2+|y|^2|\widehat{I^1\!f}|^2
		+2|y|\,\Re\big(\textsl{i}\,\widehat{Z(I^0\!f)}\cdot\overline{\widehat{I^1\!f}}\big).
		\label{4.36}
		\end{equation}
		
		Substituting the value \eqref{4.36} for $\sum\limits_{i=1}^n|\widehat{If_i}|^2$ into \eqref{4.35}, we get,
		$$
		\begin{aligned}
		\|f\|^2_{H^s_t({\R}^n;{\C}^n)}=a_n\Big[&
		\sum\limits_{i=1}^n\frac{1}{2\pi}\int\limits_{{\S}^{n-1}}\int\limits_{\xi^\bot}|y|^{2t+1}(1+|y|^2)^{s-t}|\Xi_i\widehat{I^0\!f}(y,\xi)|^2\,\md y\md \xi\\
		&+\frac{1}{2\pi}\int\limits_{{\S}^{n-1}}\int\limits_{\xi^\bot}|y|^{2t+1}(1+|y|^2)^{s-t}|\widehat{I^0\!f}(y,\xi)|^2\,\md y\md \xi\\
		&+\frac{1}{2\pi}\int\limits_{{\S}^{n-1}}\int\limits_{\xi^\bot}|y|^{2t+3}(1+|y|^2)^{s-t}|\widehat{I^1\!f}(y,\xi)|^2\,\md y\md \xi\\
		&+2\Re\Big(\frac{1}{2\pi}\int\limits_{{\S}^{n-1}}\int\limits_{\xi^\bot}|y|^{2t+2}(1+|y|^2)^{s-t}
		\textsl{i}\,\widehat{Z(I^0\!f)}(y,\xi)\,\overline{\widehat{I^1\!f}(y,\xi)}\,\md y\md \xi\Big)
		\Big].
		\end{aligned}
		$$
		On using the statement $\Xi_i\widehat{I^0\!f}=\widehat{\Xi_i(I^0\!f)}$ of Lemma \ref{L8.2}, we rewrite the latter formula as
		$$
		\begin{aligned}
		\|f\|^2_{H^s_t({\R}^n;{\C}^n)}=a_n\Big[&
		\sum\limits_{i=1}^n\frac{1}{2\pi}\int\limits_{{\S}^{n-1}}\int\limits_{\xi^\bot}|y|^{2t+1}(1+|y|^2)^{s-t}|\widehat{\Xi_i(I^0\!f)}(y,\xi)|^2\,\md y\md \xi\\
		&+\frac{1}{2\pi}\int\limits_{{\S}^{n-1}}\int\limits_{\xi^\bot}|y|^{2t+1}(1+|y|^2)^{s-t}|\widehat{I^0\!f}(y,\xi)|^2\,\md y\md \xi\\
		&+\frac{1}{2\pi}\int\limits_{{\S}^{n-1}}\int\limits_{\xi^\bot}|y|^{2t+3}(1+|y|^2)^{s-t}|\widehat{I^1\!f}(y,\xi)|^2\,\md y\md \xi\\
		&+2\Re\Big(\frac{1}{2\pi}\int\limits_{{\S}^{n-1}}\int\limits_{\xi^\bot}|y|^{2t+2}(1+|y|^2)^{s-t}
		\textsl{i}\,\widehat{Z(I^0\!f)}(y,\xi)\,\overline{\widehat{I^1\!f}(y,\xi)}\,\md y\md \xi\Big)
		\Big].
		\end{aligned}
		$$
		In view of \eqref{2.11}, this is equivalent to \eqref{4.32}.
	\end{proof}
	
	Formula \eqref{4.32} suggests the idea of introducing the norm
	\begin{equation}
	\begin{aligned}
	\|(\varphi,\psi)\|^2_{{\mathcal H}^{1,s}_t(T{\S}^{n-1})}=a_n\Big[&
	\sum\limits_{i=1}^n\|\Xi_i\varphi\|^2_{H^{s+1/2}_{t+1/2}(T{\S}^{n-1})}
	+\|\varphi\|^2_{H^{s+1/2}_{t+1/2}(T{\S}^{n-1})}\\
	&+\|\psi\|^2_{H^{s+3/2}_{t+3/2}(T{\S}^{n-1})}
	+2\Re\big(\textsl{i}\,Z\varphi,\psi\big)_{H^{s+1}_{t+1}(T{\S}^{n-1})}\Big]
	\end{aligned}
	\label{4.37}
	\end{equation}
	on the space ${\mathcal S}(T{\S}^{n-1})\oplus{\mathcal S}(T{\S}^{n-1})$.
	Let us demonstrate that the right-hand side of \eqref{4.37} is positive for $(\varphi,\psi)\neq(0,0)$. Indeed, the inequality
	\begin{equation}
	\left|(\varphi,\psi)_{H^{s+1}_{t+1}(T{\S}^{n-1})}\right|\le\|\varphi\|_{H^{s+1/2}_{t+1/2}(T{\S}^{n-1})}\|\psi\|_{H^{s+3/2}_{t+3/2}(T{\S}^{n-1})}
	\label{4.37a}
	\end{equation}
	easily follows from the definition of $H^s_t$-norms with the help of the Schwartz inequality. In particular,
	$$
	\begin{aligned}
	\left|(\textsl{i}\,Z\varphi,\psi)_{H^{s+1}_{t+1}(T{\S}^{n-1})}\right|&\le
	\|Z\varphi\|_{H^{s+1/2}_{t+1/2}(T{\S}^{n-1})}\|\psi\|_{H^{s+3/2}_{t+3/2}(T{\S}^{n-1})}\\
	&\le\frac{1}{2}\|Z\varphi\|^2_{H^{s+1/2}_{t+1/2}(T{\S}^{n-1})}+\frac{1}{2}\|\psi\|^2_{H^{s+3/2}_{t+3/2}(T{\S}^{n-1})}.
	\end{aligned}
	$$
	The inequality
	$$
	\|Z\varphi\|^2_{H^{s+1/2}_{t+1/2}(T{\S}^{n-1})}\le\sum\limits_{i=1}^n\|\Xi_i\varphi\|^2_{H^{s+1/2}_{t+1/2}(T{\S}^{n-1})}
	$$
	easily follows from definition \eqref{4.33} of the operator $Z$. Together with the previous inequality, it gives
	\begin{equation}
	2\left|\big(\textsl{i}\,Z\varphi,\psi\big)_{H^{s+1}_{t+1}(T{\S}^{n-1})}\right|\le\sum\limits_{i=1}^n\|\Xi_i\varphi\|^2_{H^{s+1/2}_{t+1/2}(T{\S}^{n-1})}
	+\|\psi\|^2_{H^{s+3/2}_{t+3/2}(T{\S}^{n-1})}.
	\label{4.38}
	\end{equation}
	On using the last inequality, we derive from \eqref{4.37}
	$$
	\|(\varphi,\psi)\|^2_{{\mathcal H}^{1,s}_t(T{\S}^{n-1})}\ge a_n	\|\varphi\|^2_{H^{s+1/2}_{t+1/2}(T{\S}^{n-1})}.
	$$
	The right-hand side of this inequality is non-negative and equals zero only if $\varphi=0$. In the latter case the right-hand side of
	\eqref{4.37} is non-negative and equals zero only if $\psi=0$.
	
	We define the Hilbert space ${\mathcal H}^{1,s}_t(T{\S}^{n-1})$ as the completion of ${\mathcal S}(T{\S}^{n-1})\oplus{\mathcal S}(T{\S}^{n-1})$ with respect to the norm \eqref{4.37}.
	
	\begin{corollary} \label{Cor4.1}
		The map
		$$
		\Sc(\Rb^{n};{\C}^n)\rightarrow{\mathcal S}(T{\S}^{n-1})\oplus{\mathcal S}(T{\S}^{n-1}),\quad f\mapsto(I^0\!f,I^1\!f)
		$$
		extends to the linear isometric embedding
		\begin{equation}
		H^s_t({\R}^n;{\C}^n)\rightarrow{\mathcal H}^{1,s}_t(T{\S}^{n-1})
		\label{4.39}
		\end{equation}
		for every $s\in\R$ and every $t>-n/2$. In other words, the equality
		\begin{equation}
		\|f\|_{H^s_t({\R}^n;{\C}^n)}=\|(I^0\!f,I^1\!f)\|_{{\mathcal H}^{1,s}_t(T{\S}^{n-1})}
		\label{4.40}
		\end{equation}
		holds for all $f\in H^s_t({\R}^n;{\C}^n)$.
	\end{corollary}
	
	The embedding \eqref{4.39} is not surjective because the function $(I^0\!f)(x,\xi)$ is odd in the second argument and the function $(I^1\!f)(x,\xi)$ is even in the second argument as is seen from \eqref{2.3}. The interesting open question is about a description of the range  of the operator \eqref{4.39}.
	
	Formula \eqref{4.40} gives the best stability estimate for the problem of recovering a vector field $f$ from the data $(I^0\!f,I^1\!f)$, although a rather exotic norm is involved. The stability estimate involving more traditional norms can be also easily derived from the Reshetnyak formula.
	
	\begin{corollary} \label{Cor4.2}
		For every $s\in\R$ and for every $t>-n/2$, the stability estimate
		\begin{equation}
		\|f\|^2_{H^s_t({\R}^n;{\C}^n)}\le 2a_n\Big(
		\sum\limits_{i=1}^n\|\Xi_i(I^0\!f)\|^2_{H^{s+1/2}_{t+1/2}}
		+\|I^0\!f\|^2_{H^{s+1/2}_{t+1/2}}
		+\|I^1\!f\|^2_{H^{s+3/2}_{t+3/2}}\Big).
		\label{4.41}
		\end{equation}
		holds for all $f\in H^s_t({\R}^n;{\C}^n)$. 
	\end{corollary}
	
	The abbreviated notation $\|\cdot\|_{H^{s'}_{t'}}$ for the norm $\|\cdot\|_{H^{s'}_{t'}(T{\S}^{n-1})}$ is used on the right-hand side of \eqref{4.41}.
	Inequality \eqref{4.41} follows from \eqref{4.32} with the help of \eqref{4.38}.

	\subsection{Reshetnyak formula for second rank symmetric tensor fields}
	
	Along with the operator $Z$ defined by \eqref{4.33}, we need two similar operators
	$$
	Q,W:{\mathcal S}(T{\S}^{n-1})\rightarrow C^\infty(T{\S}^{n-1})
	$$
	that are defined by the formulas
	$$
	(y^iy^j\Xi_i\Xi_j\hat\varphi)(y,\xi)=|y|^2\widehat{Q\varphi}(y,\xi)
	$$
	and
	$$
	(y^iX_i\hat\varphi)(y,\xi)=|y|\widehat{W\varphi}(y,\xi)
	$$
	respectively. We will sometimes write $\Xi^i$ instead of $\Xi_i$ in order to adopt our formulas to the Einstein summation rule.

	\begin{theorem} \label{Rf2}
		Given a tensor field $f\in \Sc(\Rb^{n};S^{2}\Rb^{n})$, the equality
		\begin{equation}
		\begin{aligned}
		\|f\|^2_{H^s_t({\R}^n;S^{2}\Rb^{n})}=\frac{a_n}{4}\Big[&\|I^2\!f\|^2_{H^{s+5/2}_{t+5/2}}
		+4\Im\big(I^2\!f,Z(I^1\!f)\big)_{H^{s+2}_{t+2}}
		+2\sum\limits_i\|\Xi_i(I^1\!f)\|^2_{H^{s+3/2}_{t+3/2}}\\
		&+2\|I^1\!f\|^2_{H^{s+3/2}_{t+3/2}}
		+2\|Z(I^1\!f)\|^2_{H^{s+3/2}_{t+3/2}}
		+2\Re\big(I^2\!f,W(I^0\!f)\big)_{H^{s+2}_{t+2}}\\
		&-4\Re\big(I^2\!f,I^0\!f\big)_{H^{s+3/2}_{t+3/2}}
		-2\Re\big(I^2\!f,Q(I^0\!f)\big)_{H^{s+3/2}_{t+3/2}}\\
		&-4\Re\big(\Xi^i(I^1\!f),x_iI^0\!f\big)_{H^{s+3/2}_{t+3/2}}
		+8\Im\big(Z(I^1\!f),I^0\!f\big)_{H^{s+1}_{t+1}}\\
		&+4\Im\big(I^1\!f,Z(I^0\!f)\big)_{H^{s+1}_{t+1}}
		+4\Im\big(\Xi^i(I^1\!f),Z\Xi_i(I^0\!f)\big)_{H^{s+1}_{t+1}}\\
		&+\sum\limits_i\|x_iI^0\!f\|^2_{H^{s+3/2}_{t+3/2}}
		-2\Im\big(x^iI^0\!f,Z\Xi_i(I^0\!f)\big)_{H^{s+1}_{t+1}}\\
		&-4\Re\big(W(I^0\!f),I^0\!f\big)_{H^{s+1}_{t+1}}
		+\sum\limits_{i,j}\|\Xi_i\Xi_j(I^0\!f)\|^2_{H^{s+1/2}_{t+1/2}}\\
		&+4\Re\big(\Xi^i\Xi_i(I^0\!f), I^0\!f\big)_{H^{s+1/2}_{t+1/2}}
		+\sum\limits_i\|\Xi_i(I^0\!f)\|^2_{H^{s+1/2}_{t+1/2}}
		+4n\|I^0\!f\|^2_{H^{s+1/2}_{t+1/2}}\Big]
		\end{aligned}
		\label{4.42}
		\end{equation}
		holds for every $s\in\R$ and every $t>-n/2$, where the constant $a_n$ is defined by \eqref{2.13}. 
	\end{theorem}
	
	As before, the abbreviated notation $\|\cdot\|_{H^{s'}_{t'}}$ for the norm $\|\cdot\|_{H^{s'}_{t'}(T{\S}^{n-1})}$ is used on the right-hand side of \eqref{4.42}
	
	\begin{proof}[Sketch of the proof]
		The proof follows the same line as the proof of Theorem \ref{Rf1} although all calculations are more cumbersome. We will present key formulas only.
		
		The Reshetnyak formula \eqref{2.12} for scalar functions implies
		\begin{equation}
		\|f\|^2_{H^s_t({\R}^n;S^{2}\Rb^{n})}=a_n\sum\limits_{i,j}\|If_{ij}\|^2_{H^{s+1/2}_{t+1/2}(T{\S}^{n-1})}
		=\frac{a_n}{2\pi}\sum\limits_{i,j}\int\limits_{{\S}^{n-1}}\int\limits_{\xi^\bot}|y|^{2t+1}(1+|y|^2)^{s-t}
		|\widehat{If_{ij}}(y,\xi)|^2\,\md y\md \xi.
		\label{4.43}
		\end{equation}
		
		Applying the Fourier transform to equation \eqref{4.14} and using Lemma \ref{L8.2}, we obtain the following analog of formula \eqref{4.34}:
		$$
		\widehat{If_{ij}}=\frac{1}{2}\sigma(ij)\Big( -y_i y_j\widehat{I^2\!f}
		-2\textsl{i}\,y_i\Xi_j\widehat{I^1\!f}
		-2\textsl{i}\,\xi_iy_j \widehat{I^1\!f}
		+\Xi_i \Xi_j\widehat{I^0\!f}
		-y_iX_j\widehat{I^0\!f}
		+3\xi_i \Xi_j\widehat{I^0\!f}
		+2\delta_{ij}\widehat{I^0\!f}\Big).
		$$
		This implies
		$$
		\begin{aligned}
		|\sum\limits_{i,j}\widehat{If_{ij}}(y,\xi)|^2=\frac{1}{4}&\sum\limits_{i,j=1}^n\Big|\sigma(ij)\Big( -y_i y_j\widehat{I^2\!f}(y,\xi)
		-2\textsl{i}\,y_i(\Xi_j\widehat{I^1\!f})(y,\xi)
		-2\textsl{i}\,\xi_iy_j \widehat{I^1\!f}(y,\xi)\\
		&+(\Xi_i \Xi_j\widehat{I^0\!f})(y,\xi)
		-y_i(X_j\widehat{I^0\!f})(y,\xi)
		+3\xi_i (\Xi_j\widehat{I^0\!f})(y,\xi)
		+2\delta_{ij}\widehat{I^0\!f}(y,\xi)\Big)\Big|^2.
		\end{aligned}
		$$
		Substituting this expression into \eqref{4.43}, we get,
		\begin{equation}
		\begin{aligned}
		&\|f\|^2_{H^s_t({\R}^n;S^{2}\Rb^{n})}=
		\frac{a_n}{4}\frac{1}{2\pi}\int\limits_{{\S}^{n-1}}\int\limits_{\xi^\bot}|y|^{2t+1}(1+|y|^2)^{s-t}\\
		&\times\sum\limits_{i,j=1}^n\Big|\sigma(ij)\Big( -y_i y_j\widehat{I^2\!f}(y,\xi)
		-2\textsl{i}\,y_i(\Xi_j\widehat{I^1\!f})(y,\xi)
		-2\textsl{i}\,\xi_iy_j \widehat{I^1\!f}(y,\xi)\\
		&+(\Xi_i \Xi_j\widehat{I^0\!f})(y,\xi)
		-y_i(X_j\widehat{I^0\!f})(y,\xi)
		+3\xi_i (\Xi_j\widehat{I^0\!f})(y,\xi)
		+2\delta_{ij}\widehat{I^0\!f}(y,\xi)\Big)\Big|^2
		\,\md y\md \xi.
		\end{aligned}
		\label{4.44}
		\end{equation}
		
		By pure algebraic calculations, on using the identities $|\xi|^2=1$ and $\l\xi,y\r=0$, combined with Lemma \ref{4.1}, we obtain the following analog of \eqref{4.36}:
		$$
		\begin{aligned}
		\sum\limits_{i,j=1}^n\Big|&\sigma(ij)\Big( -y_i y_j\widehat{I^2\!f}(y,\xi)
		-2\textsl{i}\,y_i(\Xi_j\widehat{I^1\!f})(y,\xi)
		-2\textsl{i}\,\xi_iy_j \widehat{I^1\!f}(y,\xi)\\
		&+(\Xi_i \Xi_j\widehat{I^0\!f})(y,\xi)
		-y_i(X_j\widehat{I^0\!f})(y,\xi)
		+3\xi_i (\Xi_j\widehat{I^0\!f})(y,\xi)
		+2\delta_{ij}\widehat{I^0\!f}(y,\xi)\Big)\Big|^2\\
		=&\ |y|^4|\widehat{I^2\!f}|^2
		+2|y|^2\sum\limits_i|\Xi_i\widehat{I^1\!f}|^2
		+2|y|^2 |\widehat{I^1\!f}|^2+2|y^i\Xi_i\widehat{I^1\!f}|^2\\
		&+\sum\limits_{i,j}|\Xi_i \Xi_j\widehat{I^0\!f}|^2
		+|y|^2\sum\limits_i|X_i\widehat{I^0\!f}|^2
		+\sum\limits_i|\Xi_i\widehat{I^0\!f}|^2
		+4n|\widehat{I^0\!f}|^2\\
		&+4|y|^2\Im(\widehat{I^2\!f}\cdot y^i\Xi_i\overline{\widehat{I^1\!f}})
		-2\Re(\widehat{I^2\!f}\cdot y^iy^j\Xi_i\Xi_j\overline{\widehat{I^0\!f}})\\
		&+2|y|^2\Re(\widehat{I^2\!f}\cdot y^iX_i\overline{\widehat{I^0\!f}})
		-4|y|^2\Re(\widehat{I^2\!f}\cdot \overline{\widehat{I^0\!f}})\\
		&+4\Im(\Xi^i\widehat{I^1\!f}\cdot y^j\Xi_j\Xi_i\overline{\widehat{I^0\!f}})
		-4|y|^2\Im(\Xi^i\widehat{I^1\!f}\cdot X_i\overline{\widehat{I^0\!f}})\\
		&+8\Im(y^i\Xi_i\widehat{I^1\!f}\cdot \overline{\widehat{I^0\!f}})
		+4\Im(\widehat{I^1\!f}\cdot y^i\Xi_i\overline{\widehat{I^0\!f}})\\
		&-2\Re(X^i\widehat{I^0\!f}\cdot y^j\Xi_j\Xi_i\overline{\widehat{I^0\!f}})
		+4\Re(\Xi^i\Xi_i\widehat{I^0\!f}\cdot \overline{\widehat{I^0\!f}})
		-4\Re(y^iX_i\widehat{I^0\!f}\cdot \overline{\widehat{I^0\!f}}).
		\end{aligned}
		$$
		Substitute this value into \eqref{4.44}
		\begin{equation}
		\begin{aligned}
		\frac{4}{a_n}\|f\|^2_{H^s_t({\R}^n;S^{2}\Rb^{n})}&= \frac{1}{2\pi}\int\limits_{{\S}^{n-1}}\int\limits_{\xi^\bot}|y|^{2t+1}(1+|y|^2)^{s-t}\bigg[
		|y|^4|\widehat{I^2\!f}(y,\xi)|^2
		+2|y|^2 |\widehat{I^1\!f}(y,\xi)|^2\\
		&+4n|\widehat{I^0\!f}(y,\xi)|^2
		-4|y|^2\Re\big(\widehat{I^2\!f}(y,\xi)\,\overline{\widehat{I^0\!f}(y,\xi)}\big)\\
		&+2|y|^2\sum\limits_i|(\Xi_i\widehat{I^1\!f})(y,\xi)|^2
		+\sum\limits_i|(\Xi_i\widehat{I^0\!f})(y,\xi)|^2\\
		&
		+\sum\limits_{i,j}|(\Xi_i \Xi_j\widehat{I^0\!f})(y,\xi)|^2
		+2|(y^i\Xi_i\widehat{I^1\!f})(y,\xi)|^2
		+|y|^2\sum\limits_i|(X_i\widehat{I^0\!f})(y,\xi)|^2\\
		&
		+4|y|^2\Im\big(\widehat{I^2\!f}(y,\xi)\, \overline{(y^i\Xi_i\widehat{I^1\!f})(y,\xi)}\big)
		-2\Re\big(\widehat{I^2\!f}(y,\xi)\,\overline{(y^iy^j\Xi_i\Xi_j\widehat{I^0\!f})(y,\xi)}\big)\\
		&+2|y|^2\Re\big(\widehat{I^2\!f}(y,\xi)\,\overline{(y^iX_i\widehat{I^0\!f})(y,\xi)}\big)
		+4\Im\big((\Xi^i\widehat{I^1\!f})(y,\xi)\, \overline{(y^j\Xi_j\Xi_i\widehat{I^0\!f})(y,\xi)}\big)\\
		&-4|y|^2\Im\big((\Xi^i\widehat{I^1\!f})(y,\xi)\, \overline{(X_i\widehat{I^0\!f})(y,\xi)}\big)
		+8\Im\big((y^i\Xi_i\widehat{I^1\!f})(y,\xi)\, \overline{\widehat{I^0\!f}(y,\xi)}\big)\\
		&+4\Im\big(\widehat{I^1\!f}(y,\xi)\, \overline{(y^i\Xi_i\widehat{I^0\!f})(y,\xi)}\big)
		-2\Re\big((X^i\widehat{I^0\!f})(y,\xi)\, \overline{(y^j\Xi_j\Xi_i\widehat{I^0\!f})(y,\xi)}\big)\\
		&+4\Re\big((\Xi^i\Xi_i\widehat{I^0\!f})(y,\xi)\, \overline{\widehat{I^0\!f}(y,\xi)}\big)
		-4\Re\big((y^iX_i\widehat{I^0\!f})(y,\xi)\, \overline{\widehat{I^0\!f}(y,\xi)}\big)
		\bigg]\,\md y\md \xi.
		\end{aligned}
		\label{4.45}
		\end{equation}
		
		There are 19 terms in the brackets on the right-hand side of \eqref{4.45} which are in one-to-one correspondence with terms on the right-hand side of \eqref{4.42}. On using Lemma \ref{L8.2} and operators $Z,Q,W$, one checks that the integral of each term in the brackets on the right-hand side of \eqref{4.45} (with the weight written before the brackets) is equal to the corresponding term on the right-hand side of \eqref{4.42}.
	\end{proof}
	
	Repeating arguments of Section 4.2, we introduce the following norm on the space
	${\mathcal S}(T{\S}^{n-1})\oplus{\mathcal S}(T{\S}^{n-1})\oplus{\mathcal S}(T{\S}^{n-1})$:
	\begin{equation}
	\begin{aligned}
	\|(\varphi,\psi,\chi)\|^2_{{\mathcal H}^{2,s}_t(T{\S}^{n-1})}&=\frac{a_n}{4}\bigg[\|\chi\|^2_{H^{s+5/2}_{t+5/2}}
	+4\Im(\chi,Z\psi)_{H^{s+2}_{t+2}}
	+2\sum\limits_i\|\Xi_i\psi\|^2_{H^{s+3/2}_{t+3/2}}\\
	&+2\|\psi\|^2_{H^{s+3/2}_{t+3/2}}
	+2\|Z\psi\|^2_{H^{s+3/2}_{t+3/2}}
	+2\Re(\chi,W\varphi)_{H^{s+2}_{t+2}}\\
	&-4\Re(\chi,\varphi)_{H^{s+3/2}_{t+3/2}}
	-2\Re(\chi,Q\varphi)_{H^{s+3/2}_{t+3/2}}\\
	&-4\Re(\Xi^i\psi,x_i\varphi)_{H^{s+3/2}_{t+3/2}}
	+8\Im(Z\psi,\varphi)_{H^{s+1}_{t+1}}\\
	&+4\Im(\psi,Z\varphi)_{H^{s+1}_{t+1}}
	+4\Im(\Xi^i\psi,Z\Xi_i\varphi)_{H^{s+1}_{t+1}}\\
	&+\sum\limits_i\|x_i\varphi\|^2_{H^{s+3/2}_{t+3/2}}
	-2\Im(x^i\varphi,Z\Xi_i\varphi)_{H^{s+1}_{t+1}}\\
	&-4\Re(W\varphi,\varphi)_{H^{s+1}_{t+1}}
	+\sum\limits_{i,j}\|\Xi_i\Xi_j\varphi\|^2_{H^{s+1/2}_{t+1/2}}\\
	&+4\Re\big(\Xi^i\Xi_i\varphi,\varphi)_{H^{s+1/2}_{t+1/2}}
	+\sum\limits_i\|\Xi_i\varphi\|^2_{H^{s+1/2}_{t+1/2}}
	+4n\|\varphi\|^2_{H^{s+1/2}_{t+1/2}}\bigg].
	\end{aligned}
	\label{4.46}
	\end{equation}
	
	\begin{lemma} \label{L4.8}
		The right-hand side of \eqref{4.46} is positive unless $(\varphi,\psi,\chi)=(0,0,0)$.
	\end{lemma}
	
	This statement was actually proved before. Indeed, the integral
	\begin{equation}
	\begin{aligned}
	&A[\varphi,\psi,\chi]=
	\frac{a_n}{4}\frac{1}{2\pi}\int\limits_{{\S}^{n-1}}\int\limits_{\xi^\bot}|y|^{2t+1}(1+|y|^2)^{s-t}
	\sum\limits_{i,j=1}^n\Big|\sigma(ij)\Big( -y_i y_j\widehat\chi(y,\xi)
	-2\textsl{i}\,y_i(\Xi_j\widehat\psi)(y,\xi)\\
	&-2\textsl{i}\,\xi_iy_j \widehat\psi(y,\xi)
	+(\Xi_i \Xi_j\widehat\varphi)(y,\xi)
	-y_i(X_j\widehat\varphi)(y,\xi)
	+3\xi_i (\Xi_j\widehat\varphi)(y,\xi)
	+2\delta_{ij}\widehat\varphi(y,\xi)\Big)\Big|^2
	\,\md y\md \xi
	\end{aligned}
	\label{4.47}
	\end{equation}
	is positive unless $\varphi=\psi=\chi=0$. The right-hand side of \eqref{4.44} is equal to $A[I^0\!f,I^1\!f,I^2\!f]$.
	After formula \eqref{4.44}, the proof of Theorem \ref{Rf2} consists of some transformations of the right-hand side of \eqref{4.44}. No specific property of the functions $(I^0\!f,I^1\!f,I^2\!f)$ is used in the transformations, i.e., the right-hand side of \eqref{4.47} can be transformed in the same manner. In this way, we demonstrate that $A[\varphi,\psi,\chi]$ is equal to the right-hand side of \eqref{4.46}.
	
	We define the Hilbert space ${\mathcal H}^{2,s}_t(T{\S}^{n-1})$ as the completion of
	${\mathcal S}(T{\S}^{n-1})\oplus{\mathcal S}(T{\S}^{n-1})\oplus{\mathcal S}(T{\S}^{n-1})$ with respect to the norm \eqref{4.46}.
	
	\begin{corollary} \label{Cor4.3}
		The map
		$$
		\Sc(\Rb^{n};S^2{\R}^n)\rightarrow{\mathcal S}(T{\S}^{n-1})\oplus{\mathcal S}(T{\S}^{n-1})\oplus{\mathcal S}(T{\S}^{n-1}),\quad f\mapsto(I^0\!f,I^1\!f,I^2\!f)
		$$
		extends to the linear isometric embedding
		$$
		H^s_t({\R}^n;{\C}^n)\rightarrow{\mathcal H}^{2,s}_t(T{\S}^{n-1})
		$$
		for every $s\in\R$ and every $t>-n/2$. In other words, the equality
		$$
		\|f\|_{H^s_t({\R}^n;S^2{\R}^n)}=\|(I^0\!f,I^1\!f,I^2\!f)\|_{{\mathcal H}^{2,s}_t(T{\S}^{n-1})}
		$$
		holds for all $f\in H^s_t({\R}^n;S^2{\R}^n)$.
	\end{corollary}
	
	Estimating scalar products on the right-hand side of \eqref{4.42} with the help of \eqref{4.37a}, we also obtain
	
	\begin{corollary} \label{Cor4.4}
		For every $s\in\R$ and for every $t>-n/2$, the stability estimate
		\begin{equation}
		\begin{aligned}
		\|f&\|^2_{H^s_t({\R}^n;S^2{\R}^n)}\leq b_n\bigg[\|I^2\!f\|^2_{H^{s+5/2}_{t+5/2}}
		+\sum\limits_i\|\Xi_i(I^1\!f)\|^2_{H^{s+3/2}_{t+3/2}}
		+\|I^1\!f\|^2_{H^{s+3/2}_{t+3/2}}
		+\|Z(I^1\!f)\|^2_{H^{s+3/2}_{t+3/2}}\\
		&+\sum\limits_i\|x_iI^0\!f\|^2_{H^{s+3/2}_{t+3/2}}
		+\sum\limits_{i,j}\|\Xi_i \Xi_j(I^0\!f)\|^2_{H^{s+1/2}_{t+1/2}}
		+ \lVert I^0\!f \rVert^2_{H^{s+1/2}_{t+1/2}}
		+\sum_{i=1}^{n}\|\Xi_i (I^0\!f)\|^2_{H^{s+1/2}_{t+1/2}}\bigg]
		\end{aligned}
		\label{4.48}
		\end{equation}
		holds for all $f\in H^s_t({\R}^n;S^2{\R}^n)$ with some constant $b_n$ depending on $n$ only.
	\end{corollary}
	
	\bigskip
	
	In principle, our approach works for every $m$, i.e., the Reshetnyak formula can be obtained which expresses the norm $\|f\|_{H^s_t({\R}^n;S^m{\R}^n)}$ through some norm of $(I^0\!f,\dots,I^m\!f)$ for a rank $m$ symmetric tensor field $f$. But the length of the formula grows fast with $m$. For example, the Reshetnyak formula contains more than 250 terms in the case of $m=3$. Unfortunately, we are not able to bring the formula into a compact form.

	\section{The two-dimensional case}\label{2D}
	
	Due to the existence of natural coordinates on $T{\S}^1={\S}^1\times\R$, all our formulas are simplified in the 2D-case. In particular, we do not need to use the operators $X_i,\Xi_i$ as well as the operators $Z,Q,W$ participating in \eqref{4.42}. Actually our work on the subject was started with considering the 2D-case.
	
	The coordinates $(p,\theta)$ on the two-dimensional manifold $T{\S}^1$ are defined by
	$$
	x=p(-\sin\theta,\cos\theta),\quad\xi=(\cos\theta,\sin\theta)\quad\mbox{for}\quad(x,\xi)\in T{\S}^1.
	$$
	The operators $X_i$ and $\Xi_i$ are easily expressed through $\frac{\partial}{\partial p}$ and $\frac{\partial}{\partial\theta}$
	\begin{equation}
	X_1=-\sin\theta\frac{\partial}{\partial p},\quad
	X_2=\cos\theta\frac{\partial}{\partial p},\quad
	\Xi_1=-\sin\theta\frac{\partial}{\partial\theta},\quad
	\Xi_2=\cos\theta\frac{\partial}{\partial\theta}.
	\label{5.1}
	\end{equation}
	In particular, the first term on the right-hand side of \eqref{4.37} can be written as
	\begin{equation}
	\sum\limits_{i=1}^2\|\Xi_i(I^0\!f)\|^2_{H^{s+1/2}_{t+1/2}(T{\S}^1)}=\Big\|\frac{\partial(I^0\!f)}{\partial\theta}\Big\|^2_{H^{s+1/2}_{t+1/2}(T{\S}^1)}.
	\label{5.2}
	\end{equation}
	
	The Fourier transform ${\mathcal S}(T{\S}^1)\rightarrow{\mathcal S}(T{\S}^1),\ \varphi\mapsto\hat\varphi$ is just the one-dimensional Fourier transform in the variable $p$ with the Fourier dual variable $q$, where $\theta$ stands as a parameter
	$$
	\hat\varphi(q,\theta)=\frac{1}{\sqrt{2\pi}}\int\limits_{-\infty}^\infty e^{-\textsl{i}\,q\,p}\,\varphi(p,\theta)\,\md p.
	$$
	The {\it Hilbert transform}
	$H:H^{s}_{t}(T{\S}^1)\rightarrow H^{s}_{t}(T{\S}^1)$
	is defined by
	$$
	\widehat{H \vp}(q,\theta)  = \mathrm{sgn}(q)\,\hat{\vp}(q,\theta).
	$$
	As is seen from \eqref{4.33} and \eqref{5.1},
	\begin{equation}
	Z=H\frac{\partial}{\partial\theta}.
	\label{5.3}
	\end{equation}
	
	With the help of \eqref{5.2} and \eqref{5.3}, the Reshetnyak formula \eqref{4.32} for vector fields takes the following form in the two-dimensional case:
	$$
	\begin{aligned}
	\|f\|^2_{H^s_t({\R}^2;{\C}^2)}=\frac{1}{2}\Big[&\|\partial_\theta(I^0\!f)\|^2_{H^{s+1/2}_{t+1/2}(T{\S}^1)}
	+\|I^0\!f\|^2_{H^{s+1/2}_{t+1/2}(T{\S}^1)}\\
	&+\|I^1\!f\|^2_{H^{s+3/2}_{t+3/2}(T{\S}^1)}
	+2\Re\big(\textsl{i}H\partial_\theta(I^0\!f),I^1\!f\big)_{H^{s+1}_{t+1}(T{\S}^1)}\Big],
	\end{aligned}
	$$
	and the stability estimate \eqref{4.41} takes the form
	$$
	\|f\|^2_{H^s_t({\R}^2;{\C}^2)}\le\|\partial_\theta(I^0\!f)\|^2_{H^{s+1/2}_{t+1/2}(T{\S}^1)}
	+\|I^0\!f\|^2_{H^{s+1/2}_{t+1/2}(T{\S}^1)}+\|I^1\!f\|^2_{H^{s+3/2}_{t+3/2}(T{\S}^1)}.
	$$
	
	Formulas \eqref{4.42} and \eqref{4.48} admit similar simplifications in the 2D-case. Omitting details, we present the result. the Reshetnyak formula \eqref{4.42} for second rank tensor fields takes the following form in the two-dimensional case:
	$$
	\begin{aligned}
	8\|f&\|^2_{H^s_t({\R}^2;S^2{\R}^2)}=
	\|I^2\!f\|^2_{H^{s+5/2}_{t+5/2}}
	+4\|\partial_\theta(I^1\!f)\|^2_{H^{s+3/2}_{t+3/2}}
	+2\|I^1\!f\|^2_{H^{s+3/2}_{t+3/2}}\\
	&+\|\partial^2_\theta(I^0\!f)\|^2_{H^{s+1/2}_{t+1/2}}
	-2\|\partial_\theta(I^0\!f)\|^2_{H^{s+1/2}_{t+1/2}}
	+\|pI^0\!f\|^2_{H^{s+3/2}_{t+3/2}}
	+8\|I^0\!f\|^2_{H^{s+1/2}_{t+1/2}}\\
	&-4\Re\big({\textsl i}H(I^2\!f),\partial_\theta(I^1\!f)\big)_{H^{s+2}_{t+2}}
	+2\Re\big({\textsl i}H(I^2\!f),pI^0\!f\big)_{H^{s+2}_{t+2}}
	-2\Re\big(I^2\!f,\partial^2_\theta(I^0\!f)\big)_{H^{s+3/2}_{t+3/2}}\\
	&-4\Re\big(I^2\!f,I^0\!f\big)_{H^{s+3/2}_{t+3/2}}
	-4\Re\big(\partial_\theta(I^1\!f),pI^0\!f\big)_{H^{s+3/2}_{t+3/2}}
	-4\Re\big({\textsl i}H\partial_\theta(I^1\!f),\partial^2_\theta(I^0\!f)\big)_{H^{s+1}_{t+1}}\\
	&+4\Re\big({\textsl i}H(I^1\!f),\partial_\theta(I^0\!f)\big)_{H^{s+1}_{t+1}}
	-2\Re\big({\textsl i}H\partial^2_\theta(I^0\!f),pI^0\!f\big)_{H^{s+1}_{t+1}}
	+4\Re\big({\textsl i}H(pI^0\!f),I^0\!f\big)_{H^{s+1}_{t+1}}.
	\end{aligned}
	$$
	Here the shorter notation $H^{s'}_{t'}$ is used instead of $H^{s'}_{t'}(T{\S}^1)$ on the right-hand side.
	The corresponding stability estimate is
	$$
	\begin{aligned}
	\|f\|^2_{H^s_t({\R}^2;S^2{\R}^2)}\le 6&\Big[
	\|I^2\!f\|^2_{H^{s+5/2}_{t+5/2}}
	+\|\partial_\theta(I^1\!f)\|^2_{H^{s+3/2}_{t+3/2}}
	+\|I^1\!f\|^2_{H^{s+3/2}_{t+3/2}}\\
	&+\|pI^0\!f\|^2_{H^{s+3/2}_{t+3/2}}
	+\|\partial^2_\theta(I^0\!f)\|^2_{H^{s+1/2}_{t+1/2}}
	+\|\partial_\theta(I^0\!f)\|^2_{H^{s+1/2}_{t+1/2}}
	+\|I^0\!f\|^2_{H^{s+1/2}_{t+1/2}}
	\Big].
	\end{aligned}
	$$
	
	Even in the 2D-case, the Reshetnyak formula for rank 3 symmetric tensor fields is too long to be presented here.


\begin{thebibliography}{10}
		
		\bibitem{AM}
		Anuj Abhishek and Rohit~Kumar Mishra.
		\newblock Support theorems and an injectivity result for integral moments of a
		symmetric m-tensor field.
		\newblock https://arxiv.org/abs/1704.02010.
		\bibitem{LS}
		W. Lionheart and V. Sharafutdinov.
		\newblock Reconstruction algorithm for the linearized polarization tomography problem with incomplete data.
		\newblock in {\em Imaging Microstructures: Mathematical and Computational Challenges}, Ed. Habib Ammari and Hyeonbae Kang,
		Contemporary Mathematics {\bf 494} (2009), 137-160.
		
		\bibitem{mb}
		Vladimir~A. Sharafutdinov.
		\newblock {\em Integral geometry of tensor fields}.
		\newblock Inverse and Ill-posed Problems Series. VSP, Utrecht, 1994.
		
		\bibitem{Sh3}
		Vladimir~A. Sharafutdinov.
		\newblock The {R}eshetnyak formula and {N}atterer stability estimates in tensor
		tomography.
		\newblock {\em Inverse Problems}, 33(2):025002, 20, 2017.
		
		\bibitem{SW}
		V. Sharafutdinov and J. Wang.
		\newblock Tomography of small residual stresses,
		\newblock Inverse Problems {\bf 28} (2012), doi: 10.1088/0266-5611/28/6/065017, (17 pp).
		
	\end{thebibliography}
\end{document}